\documentclass[12pt,reqno ]{amsart}

\usepackage{amssymb}
 \usepackage{amsthm}
\usepackage{amsmath}
\usepackage{amsrefs}
\usepackage{tikz}
\usepackage{hyperref}
\usetikzlibrary{decorations.pathreplacing}
\newtheorem{thm}{Theorem}[section]

\newtheorem{lem}[thm]{Lemma}

\theoremstyle{definition}
\newtheorem*{rem}{Remark}
\DeclareMathOperator{\card}{card}
\DeclareMathOperator{\bdd}{bdd}
\DeclareMathOperator{\s}{S}

\begin{document}

\title[Schatten $p$ class commutators]{Schatten $p$ class commutators on the weighted Bergman space $L^2 _a (\mathbb{B}_n, dv_\gamma)$ for $\frac{2n}{n + 1 + \gamma} < p < \infty$.}


\author{Joshua Isralowitz}
\thanks{Supported by an Emmy-Noether grant of Deutsche Forschungsgemeinschaft}
\email{jbi2@uni-math.gwdg.de}
\address{Mathematisches Institut \\
Georg-August Universit\"{a}t  G\"{o}ttingen \\
Bunsenstra$\ss$e 3-5 \\
D-37073, G\"{o}ttingen \\
Germany}

\begin{abstract}
Let $P_\gamma$ be the orthogonal projection from the space $L ^2 (\mathbb{B}_n, dv_\gamma)$ to the standard weighted Bergman space $L_a ^2 (\mathbb{B}_n, dv_\gamma)$.  In this paper, we characterize the Schatten $p$ class membership of the commutator $[M_f, P_\gamma]$ when $\frac{2n}{n + 1 + \gamma} < p < \infty$.  In particular, if $\frac{2n}{n + 1 + \gamma} < p < \infty$, then we show that $[M_f, P_\gamma]$ is in the Schatten $p$ class if and only if the mean oscillation MO${}_\gamma (f)$ is in $ L^p(\mathbb{B}_n, d\tau)$ where $d\tau$ is the M\"{o}bius invariant measure on $\mathbb{B}_n.$  This answers a question recently raised by K. Zhu.
\end{abstract}


\maketitle


\section{Introduction}

Let $\mathbb{B}_n \subset \mathbb{C}^n$ be the unit ball in $\mathbb{C}^n$ and let Hol$(\mathbb{B}_n)$ denote the space of holomorphic functions on $\mathbb{B}_n$. For $z, w \in \mathbb{C}^n$ with $z = (z_1, \ldots, z_n)$ and $w = (w_1, \ldots, w_n)$, let \begin{align} z \cdot w = z_1 \overline{w_1} +  \cdots + z_n \overline{w_n}. \nonumber \end{align} Let $L_a ^2(\mathbb{B}_n, dv_\gamma)$ denote the Bergman space Hol$(\mathbb{B}_n) \cap L^2(\mathbb{B}_n, dv_\gamma)$, where here (and throughout the paper) $dv$ is the ordinary Lebesgue volume measure on $\mathbb{C}^n$, and the probability measure $dv_\gamma $ for $\gamma > -	1$ is defined by \begin{align} dv_\gamma(z) = c_\gamma (1 - |z|^2)^{n +  1 + \gamma } dv(z) \nonumber \end{align} where $c_\gamma$ is the  normalizing constant $C_\gamma = \Gamma(n  + 1 + \gamma)/ \pi^n  \Gamma(\gamma + 1)$ (see \cite{Z4} for a general reference on the weighted Bergman space of the unit ball).

It is well known that the orthogonal projection $P_\gamma$ from $L ^2 (\mathbb{B}_n, dv_\gamma)$ onto $L_a ^2 (\mathbb{B}_n, dv_\gamma)$ is an integral operator on $L^2 (\mathbb{B}_n, dv_\gamma)$ whose kernel is the reproducing kernel \begin{align} K_\gamma (z, w) =  \frac{1}{(1 - z\cdot w)^{n +  1 + \gamma}} \nonumber \end{align} of $L_a ^2 (\mathbb{B}_n, dv_\gamma)$.  Given an $f \in L^2 (\mathbb{B}_n, dv_\gamma)$, let $[M_f, P_\gamma]$ denote the first order commutator on  $L^2 (\mathbb{B}_n, dv_\gamma)$.  Closely related to the commutator $[M_f, P_\gamma]$ is the Hankel operator $H_f = (I - P_\gamma) M_f P_\gamma$ on $L^2 (\mathbb{B}_n, dv_\gamma)$.  Because of the easily verified relation \begin{align} [M_f, P_\gamma] = H_f - H_{\overline{f}} ^*, \nonumber \end{align}  the study of the commutator $[M_f, P_\gamma]$ is equivalent to the study of the simultaneous study of the Hankel operators $H_f$ and $H_{\overline{f}}$ (see \cite{Z2} and the reference therein for results relevant to the boundedness, compactness, and Schatten class membership of $H_f$ and $H_{\overline{f}}$).

If $H$ is a separable Hilbert space, then recall that a bounded operator $T$ on $H$ is in the Schatten $p$ class (usually denoted by S${}_p$) if $(T^*T)^\frac{p}{2}$ is trace class.  This paper will discuss the Schatten class membership of the commutator $[M_f, P_\gamma]$, but we will discuss the relevant literature before we state our main result.  First, let $k_z ^\gamma(w) $ be the normalized reproducing kernel of $L_a ^2 (\mathbb{B}_n, dv_\gamma)$ given by \begin{align} k_z ^\gamma(w) = \frac{(1 - |z|^2)^\frac{n + 1 + \gamma}{2}}{(1 - w \cdot z)^{n + 1 + \gamma} } \nonumber \end{align}  Moreover, for $f \in L ^2 (\mathbb{B}_n, dv_\gamma)$, let $B_\gamma (f)$ be the Berezin transform on $\mathbb{B}_n$ defined by \begin{align} B_\gamma (f)(z) & = \int_{\mathbb{B}_n} f(w) |k_z ^\gamma (w)|^2 \, dv_\gamma (w) \nonumber \\ & = \int_{\mathbb{B}_n} f(w) \frac{(1 - |z|^2)^{n + 1 + \gamma}}{|1 - w \cdot z |^{2(n + 1 + \gamma)}}  \, dv_\gamma (w) \nonumber \end{align}  and let the mean oscilation MO${}_\gamma (f)$ be the function on $\mathbb{B}_n$ defined by \begin{align}
\text{MO}_\gamma (f)(z) = \Big \{B_\gamma (|f|^2)(z) - | B_\gamma (f)(z) |^2 \Big \}^\frac{1}{2}. \nonumber \end{align}
Moreover, let $\text{BMO}_\partial$ be the space of all $f \in L^2(\mathbb{B}_n, dv_\gamma)$ where $MO_\gamma(f)$ is bounded.  Note that $\text{BMO}_\partial$ as a vector space is in fact independent of $\gamma$ and note that $[M_f, P_\gamma]$ is bounded if and only if $f \in \text{BMO}_\partial$ (see \cite{Z2}, chap. $8$ for more details).  

In \cite{Z1}, it was proved that $[M_f, P_\gamma]$ is in the Schatten $p$ class for $p \geq 2$ if and only if MO${}_\gamma (f) \in L^2 (\mathbb{B}_n, d\tau)$ where $d\tau$ is the M\"{o}bius invariant measure on $\mathbb{B}_n$ given by $d\tau (z) = (1 - |z|^2)^{- n - 1} dv(z)$.  Moreover, it was proved in \cite{X1} that the same characterization of Schatten class commutators holds when $\max \Big \{ 1, \frac{2n}{n + 1 + \gamma} \Big \} < p \leq 2$ (where $p = 1$ is allowed when $\frac{2n}{n + 1 + \gamma} < 1 $).  Note that only the $\gamma = 0$ case was proven for both these results in \cite{Z1, X1}, but the extension to arbitrary $\gamma > -1$ is routine.

It is easy to see why the ``cut-off" term in the previous result is natural.  In particular, it is elementary to see that \begin{align} \text{MO}_\gamma (f) (z) & \geq \Big \{ \int_{\mathbb{B}_n} |f - B_\gamma (f)(z)|^2 \frac{(1 - |z|^2)^{n + 1 + \gamma}}{(1 + |z|^2 )^{2(n + 1 + \gamma)}} \, dv_\gamma  \Big \} ^\frac{1}{2} \nonumber \\ & \geq 2^{- (n + 1 + \gamma)}  ( 1 - |z|^2)^\frac{(n + 1 + \gamma )}{2} \underset{\alpha \in \mathbb{C}}{\inf} \Big \{\int_{\mathbb{B}_n} |f - \alpha |^2 \, dv_\gamma \Big \}^\frac{1}{2}. \nonumber \end{align}  But then $\text{MO}_\gamma (f) \in L^p (\mathbb{B}_n, d\tau)$ when $p \leq \frac{2n}{n + 1 + \gamma}$ only if \begin{align} \underset{\alpha \in \mathbb{C}}{\inf} \int_{\mathbb{B}_n} |f - \alpha |^2 \, dv_\gamma  = 0, \nonumber \end{align} which is true only if $f$ is identically constant a.e. on $\mathbb{B}_n$.

Thus, when $p \leq \frac{2n}{n + 1 + \gamma}$, the mean oscillation can not be used to characterize $\text{S}_p$ commutators $[M_f, P_\gamma]$.  However, when $\gamma > n - 1$, we have that $\frac{2n}{n + 1 + \gamma} < 1$. This leads to the question (first raised in \cite{Z2}, p. 227 when $n = 1$) of whether the result in \cite{X1} holds for the general range $p > \frac{2n}{n + 1 + \gamma}$. Note that the (very easy and short) proof of sufficiency in \cite{X1}, p. 915 holds when $p \leq 1$,  so that the only non-trivial portion of this question is whether necessity holds.  The main result in this paper will be an affirmative answer to this question.  In particular, we will prove the following:

\begin{thm}  Let $f \in \text{BMO}_\partial$ and let $\frac{2n}{ n + 1 + \gamma} < p < 1$.  Then the commutator $[M_f, P_\gamma] \in \text{S}_p$ if and only if MO${}_\gamma (f) \in L^p (\mathbb{B}_n, d\tau)$ where $d\tau$ is the M\"{o}bius invariant measure on $\mathbb{B}_n$. \end{thm} \noindent In the last section of this paper, we will briefly discuss the case when $0 < p \leq \frac{2n}{ n + 1 + \gamma}$ and formulate a conjecture regarding the characterization of Schatten $p$ class commutators when $p$ is below this cut-off. Note that Theorem $1$ is in fact independent of $\gamma$ in the sense that if Theorem $1$ is true for some $\gamma > -1$ then it is true for all $\gamma > -1$ (see Theorem $6.1$) in section $6$

Finally in this introduction we will outline the rest of the paper.  The next section will discuss the Bergman tree $\mathcal{T}_n$ with tree parameter $\lambda$ from \cite{ARS}, which decomposes the ball $\mathbb{B}_n$ into disjoint ``top-half Carleson'' like sets and assigns a tree structure to this decomposition of $\mathbb{B}_n.$  As will be see in the subsequent sections, this tree structure will simplify the notation used (when compared to the notation used in \cite{X1}), and in fact could be used to simplify the notation used in \cite{X1} even when $n = 1$.   In section three, we provide a discretized version of the condition MO$_\gamma (f)  \in L^p(\mathbb{B}_n, d\tau)$  when $p > \frac{2n}{n + 1 + \gamma}$. Note that a similar discretization appeared in \cite{X1} when $\gamma = 0$ and $p \geq 1$.  In section four, we will prove some important lemmas that will be needed in the proof of Theorem $1.1$, including a crucial ``reverse Cauchy-Schwarz inequality" that is valid for $\lambda$ small enough, and we will prove Theorem $1.1$ in section five. As stated before, in the last section we will discuss and formulate a conjecture for the case $0 < p \leq \frac{2n}{n + 1 + \gamma}$. In the last section we will also very briefly discuss ``Bergman metric Besov space'' on the ball.

 Note that Theorem $1.1$ was proven in \cite{I} in the context of the weighted Fock space $F_\alpha ^2 (\mathbb{C}^n)$ for the full range $0 < p < 1$ (only $\alpha = 1$ was considered in \cite{I}, but the extension to general $\alpha > 0$ is routine.) However, the details of the proof of Theorem $1.1$ are more involved and are considerably more messy than the details of the proof of the corresponding result in \cite{I}.  There are two simple reasons for this difference in details.  First is the obvious fact that the ball $\mathbb{B}_n$ with the Bergman metric has a much more complicated geometry than $\mathbb{C}^n$ with the Euclidean metric has.  Second, if $k_z^ {(\alpha)} (w)$ is the normalized reproducing kernel of the weighted Fock space $F_\alpha ^2 (\mathbb{C}^n)$ (with weight $\alpha > 0$), then  \begin{align} |k_z ^{(\alpha)} (w)|^2 e^{- \alpha |w|^2} = e^{- \alpha |z - w|^2 }, \nonumber \end{align} which decays exponentially fast when the Euclidean distance between $z$ and $w$ grows.  On the other hand, \begin{align}
|k_z ^\gamma (w)|^2 = \frac{(1 - |z|^2)^{n + 1 + \gamma}}{|1 - w \cdot z |^{2(n + 1 + \gamma)}},\nonumber \end{align} which has polynomial decay when the Bergman distance between $z$ and $w$ grows.

Given these facts, it is perhaps not surprising that Theorem $1.1$ should be true for large enough $\gamma$ (since the decay in $|k_z ^\gamma (w)|^2$ grows as $\gamma$ increases.) What is rather interesting is that the condition $p > \frac{2n}{n + 1 + \gamma}$ provides just enough decay for our arguments to work, in the sense that numerous estimates throughout the paper completely fall apart when $0 < p \leq \frac{2n}{n + 1 + \gamma}$.

Finally, we remark that Schatten $p$ class Hankel operators for $0 < p \leq 1$ on the classical Hardy space were characterized in \cite{S} using ideas that are somewhat similar to those of this paper (the same characterization was also independently obtained in \cite{P} using completely different ideas.)  

\section{The Bergman tree}

For any $z, w \in \mathbb{B}_n$, we will let $d(z, w)$ denote the Bergman distance between $z$ and $w$, which is defined by \begin{align} d(z, w) = \frac{1}{2} \ln \left(\frac{1 + |\varphi_z (w)|}{1 - |\varphi_z(w)|}\right) \nonumber \end{align} where $\varphi_z$ is the involutive automorphism of $\mathbb{B}_n$ that interchanges $z$ and $0$.  We will frequently use the following well-known equality involving $\varphi_z(w)$: \begin{align} 1 - |\varphi_z (w)|^2  = \frac{(1 - |z|^2)(1 - |w|^2)}{|1 - w \cdot z|^2} \nonumber \end{align} (see \cite{Z4} for more details).

  As stated in the introduction, this section is devoted to the definition and properties of the Bergman tree $\mathcal{T}_n$ from \cite{ARS}. To define $\mathcal{T}_n$, we use the following simple lemma whose proof is elementary (or can be found in  \cite{Z4}.)

\begin{lem} For a fixed $r > 0$, let $X = \{z \in \mathbb{B}_n : d(0, z) = r\}$. For any fixed $\lambda > 0$, there exists a finite subset $E = \{z_j \} _{j = 1}^{J} $ and corresponding Borel subsets $Q_j $ of $X$ satisfying \[\begin{array}{lll}
& X = \cup_{j = 1}^{J} Q_j  & \\
 & Q_i  \cap Q_j  = \emptyset, \hspace{5mm} i \neq j & \\
& B(z_j, \frac{\lambda}{2} ) \subseteq Q_j \subseteq B(z_j, 2 \lambda),  \hspace{5mm} j = 1, \ldots, J & \end{array} \] where $B(z, s) \subseteq X$ denotes the restriction to $X$ of the open Bergman ball at $z \in X$ with radius $s$. \end{lem}

Fix some $\lambda > 0$ (which is a small constant to be determined later) and for each $N \in \mathbb{N}$, apply the previous lemma to the Bergman spheres $S_{\lambda N} = \{z  : d(0, z) =
\lambda N \}$ to get the subsets $\{Q_j ^N\}_ {j = 1}^{J_{\lambda N}}$ of $S_{\lambda N}$ .   For $ N \geq 1, j \in \{1, \ldots, J_{\lambda N}\}$, let   \begin{align} K_j ^N = \{z \in \mathbb{B}_n : \lambda N \leq d(0, z) < \lambda (N + 1) ,  P_{\lambda N} z \in Q_j ^{J_{\lambda N}}\} \nonumber \end{align} where for any $r > 0$ and $z \in \mathbb{B}_n,  P_r z$ denotes the radial projection onto  the Bergman sphere $S_r$. Define the corresponding points $c_j ^N \in K_j ^N$ by \begin{align} c_j^N = P_{\lambda (N + \frac{1}{2})} (z_j ^N), \nonumber \end{align} and call $c_j ^N$ the center of $K_j ^N$.

We define a tree structure on the collection \begin{align} \mathcal{T}_n = \{K_j ^N\}_{N \geq 0, \ j \in \{1, \ldots, J_{\lambda N}\}} \nonumber \end{align} by declaring that $K_i ^{N + 1}$ is a child of $K_j ^N$ if the projection $P_{\lambda N} (c_i ^{N + 1})$ onto the sphere $S_{\lambda N} $ lies in $Q_j ^N$.  In the case $N = 0$, we declare that every set $K_j ^1$ is a child of the ``root'' $K_1 ^0$, which will be denoted by $K_o$.  We typically write $\alpha, \beta,$ etc. to denote the pair $(N, j)$ corresponding to the element $K_j ^N$ of the tree $\mathcal{T}_n$, and we will often use $c_\alpha$ to denote the center $c_\alpha = c_j ^N$ of the set $K_\alpha = K_j ^N$. Also, if $K_\alpha = K_j ^N$ for some $j$, then we write $d(\alpha) = N$. We  write $K_\beta \geq K_\alpha$ if $K_\beta$ is the child of $K_\alpha$, and we also write $\beta \geq \alpha$ if $K_\beta \geq K_\alpha$.  We also let $\mathcal{C} (\alpha) = \bigcup_{\beta \geq \alpha} \beta$ and $\mathcal{C}^\ell (\alpha) = \{\beta : \beta \geq \alpha \text{ and } d(\beta) = d(\alpha) + \ell\}$.  Moreover, we will call $\lambda$ the Bergan tree parameter.       

For two quantities $A$ and $B$ we will write $A \lesssim B$ to mean that there exists a universal constant $C > 0$ that depends only on the dimension $n$, the weight $\gamma$, the Bergman tree parameter $\lambda$, and $p$,  where $ A \leq C B \nonumber .$ If $A \lesssim B$ and $B \lesssim A$ then we write $A \approx B$. Now, since $d(0, z) = \tanh ^{-1} |z| ,$ it is easy to see that $1 - |z|^2 \approx e^{-2\lambda  d(\alpha)}$ if $z \in K_\alpha$ with $\alpha \in \mathcal{T}_n$.

The following properties of the Bergman tree $\mathcal{T}_n$ were first proved in \cite{ARS} and will be used throughout the paper
\begin{lem}  The tree $\mathcal{T}_n$ with parameter $\lambda$ satisfies the following properties.
\begin{list}{}
\item $a)  $ There exists  $C_1$ and  $C_2 $ independent of  $\alpha$  such that $D(c_\alpha, C_1) \subseteq K_\alpha \subseteq D(c_\alpha, C_2)$ where $D(z, r) \subset \mathbb{B}_n$ is a Bergman ball with center $z$ and radius $r.$
 \item $b)  $ For any $ \gamma \in \mathbb{R}, $ we have that $v_\gamma(K_\alpha) \approx e^{-2 \lambda d(\alpha)(n + 1 + \gamma)}. $
\item $c)  $ For any $\alpha \in \mathcal{T}_n,$ we have that $\card \mathcal{C} ^\ell (\alpha) \lesssim e^{2 n \ell \lambda}. $  \end{list}\end{lem}

In this paper, we will only require that $\lambda$ be ``small enough'' (which will be made more precise in section four.) Note that when $n = 1$, setting $\lambda = \frac{ln 2}{  2^N}$ for a non-zero integer $N$ will allow us to take $\mathcal{T}_n$ as the standard Whitney decomposition of $\mathbb{D}$ into dyadic ``top-half'' Carleson squares (see figure $1$ for the case $\lambda = \frac{\ln 2}{2}$.  For graphical purposes, the radial decomposition of $\mathbb{D}$ is not precisely drawn to scale.)  \begin{center}
\begin{tikzpicture}[scale = .5]
 \tikzstyle{every node}=[font=\Large]

\filldraw[color = black!10] (80mm, 0mm) arc (0:360 :80mm);
 \foreach \x in {36mm, 55.8mm, 66.69mm, 72.68mm, 75.97mm, 77.78mm, 78.7755mm, 79.323mm}
\draw[thin] (\x,0mm) arc (0:360:\x);

\draw[ultra thick] (80mm, 0mm) arc (0:360 :80mm);


\foreach \y in { .5, 1 }
{ \draw[thin] (0 + 360 * \y : 36mm ) -- (0 + 360 * \y : 55.8mm ); }


\foreach \y in {1, 2, 3, 4 }
{ \draw[thin] (0 + 360 * \y /4: 55.8mm ) -- (0 + 360 * \y /4 : 66.69mm  ); }

\foreach \y in {1, 2, 3, 4, 5, 6, 7, 8 }
{ \draw[thin] (0 + 360 * \y /8: 66.69mm ) -- (0 + 360 * \y /8: 72.68mm ); }

\foreach \y in {1, ..., 16 }
{ \draw[thin] (0 + 360 * \y /16: 72.68mm ) -- (0 + 360 * \y /16: 75.97mm ); }

\foreach \y in {1, ..., 32 }
{ \draw[thin] (0 + 360 * \y /32: 75.97mm ) -- (0 + 360 * \y /32: 77.78mm ); }

\foreach \y in {1, ..., 64 }
{ \draw[thin] (0 + 5.625 * \y : 77.78mm ) -- (0 + 5.625 * \y: 78.78mm ); }

\foreach \y in {1, ..., 128 }
{ \draw[thin] (0 + 2.8125 * \y : 78.78mm ) -- (0 + 2.8125 * \y: 79.323mm ); }

\draw (180 :90mm) node{$\partial \mathbb{D}$}  ;

\end{tikzpicture}
\smallskip

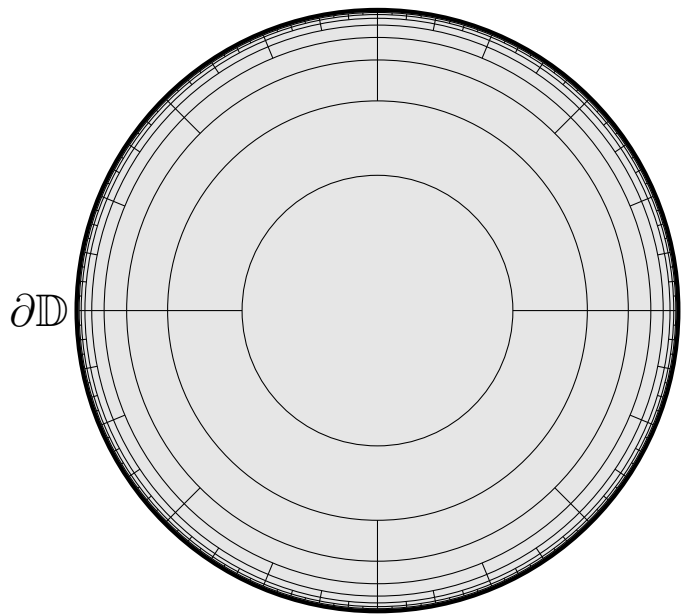
\noindent Figure $1$: The Bergman tree $\mathcal{T}_n$ when $n = 1$ consisting of dyadic ``top-half'' Carleson squares.
\end{center}

\noindent Note that we obviously do not have such an explicit contruction of $\mathcal{T}_n$ when $n > 1$.  However, the $n = 1$ situation provides us with very useful intuition (particularly when arguments involve the children or parents of a typical element $\nu \in \mathcal{T}_n$, see figure $2$.)  Thus, the reader should keep the $n = 1$ case in mind (and in particular should keep figure $2$ in mind) when reading the rest of the paper.

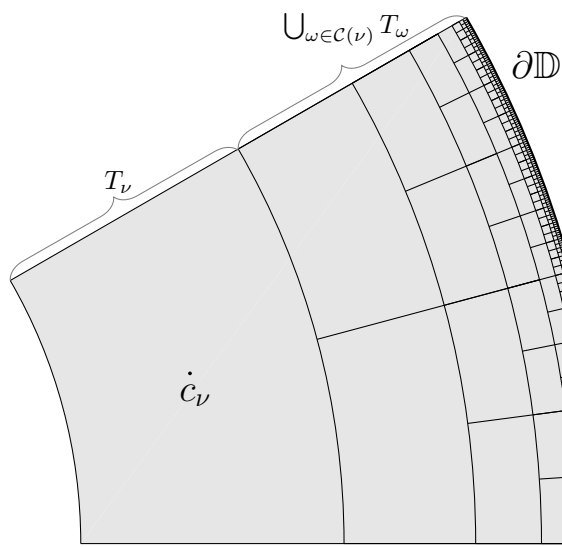
\begin{center}
\begin{tikzpicture}[scale = 1.75]
 \tikzstyle{every node}=[font=\large]

\fill[color = black!10] 
(40mm, 0mm) arc (0:30:40mm) -- (30 : 40mm) -- (30: 80mm) ;
\fill[color = black!10] 

 (80mm, 0mm) arc (0:30:80mm) -- ( 0 : 40mm) -- (0: 80mm) ;

\foreach \y in {0, ..., 1}
\draw[thin] (30 * \y : 40mm) -- (30 * \y : 77.5mm);

\foreach \y in {0, ..., 2}
\draw[thin] (15 * \y : 60mm) -- (15 * \y : 77.5mm);

\foreach \y in {0, ..., 4}
\draw[thin] (7.5 * \y : 70mm) -- (7.5 * \y : 77.5mm);

\foreach \y in {0, ..., 8}
\draw[thin] (3.75 * \y : 75mm) -- (3.75 * \y : 77.5mm);

\foreach \y in {0, ..., 16}
\draw[thin] (1.875 * \y : 77.5mm) -- (1.875 * \y : 78.75mm);

\foreach \y in {0, ..., 32}
\draw[thin] (0.9375 * \y : 78.75mm) -- (0.9375 * \y : 79.375mm);

\foreach \y in {0, ..., 64}
\draw[thin] (0.46875 * \y : 78.75mm) -- (.46875 * \y : 79.375mm);


\foreach \y in {0, ..., 128}
\draw[thin] (0.234375 * \y : 79.375mm) -- (0.234375 * \y : 79.6875mm);


\foreach \y in {0, ..., 256}
\draw[thin] (0.1171875 * \y : 79.6875mm) -- (0.1171875 * \y : 80mm);

\draw[thin] (40mm,0mm) arc (0:30:40mm);
\draw[thin] (60mm,0mm) arc (0:30:60mm);
\draw[thin] (70mm, 0mm) arc (0:30:70mm);
\draw[thin] (75mm, 0mm) arc (0:30:75mm);
\draw[thin] (77.5mm, 0mm) arc (0:30:77.5mm);
\draw[thin] (78.75mm, 0mm) arc (0:30:78.75mm);
\draw[thin] (79.375mm, 0mm) arc (0:30:79.375mm);
\draw[thin] (79.6875mm, 0mm) arc (0:30:79.6875mm);
\draw[thick](80mm, 0mm) arc (0: 30: 80mm);

\draw plot[only marks,mark=*, mark size = 0.1mm] coordinates{(15:50mm)};
\draw (13 :50mm) node{$c_\nu$}  ;
\draw (26 : 83mm) node {$\partial \mathbb{D}$};

\draw [gray,decorate,decoration={brace,amplitude=8pt}]
  (30 : 40mm) -- (30 : 60mm) 
node [black,midway,above=4pt,xshift=-2pt] {\footnotesize $T_\nu$};

\draw [gray,decorate,decoration={brace,amplitude=8pt}]
  (30 : 60mm) -- (30 : 80mm) 
node [black,midway,above=10pt,xshift=-2pt] {\footnotesize $\bigcup_{\omega \in \mathcal{C}(\nu)} T_\omega$};

\end{tikzpicture}
 \smallskip

\noindent Figure $2$: 	A typical set $T_\nu$ for $\nu \in \mathcal{T}_n$ (when $n = 1$) with center $c_\nu$ and the union $\bigcup_{\omega \in \mathcal{C}(\nu)} T_\omega$ of its children.  
\end{center}

We will now prove some preliminary results that will be used in our subsequent arguments.   For $z, w \in \partial \mathbb{B}_n$ we will let $\beta(z, w)$ be the non-isotropic distance between $z$ and $w$ given by \begin{align} \beta(z, w) = \left| 1 - z \cdot w \right|^\frac{1}{2} \nonumber \end{align} We will use the following lemma on a number of occasions \begin{lem} If $\alpha, \alpha' \in \mathcal{T}_n$ with $d(c_\alpha, c_{\alpha'}) > M$ for some $M > 0$ and $d(\alpha) = d(\alpha') \geq 1$, then $(e^M - 1)^\frac{1}{2}  e^{-\lambda d(\alpha)} \lesssim \beta(c_\alpha/|c_\alpha|, c_{\alpha '} /|c_{\alpha '}|) $ \end{lem}

\begin{proof} If $d(c_\alpha, c_{\alpha'}) > M$, then  \begin{align} 2M  <  \ln \left(\frac{1 + |\varphi_{c_\alpha} (c_{\alpha'})|}{1 - |\varphi_{c_\alpha} (c_{\alpha'})|}\right) & < \ln \left(\frac{4}{1 - |\varphi_{c_\alpha} (c_{\alpha'})|^2} \right)\nonumber \\ & =  \ln  \left(\frac{4 \left|1 -  c_\alpha \cdot c_{\alpha'}  \right|^2}{(1 - |c_\alpha|^2)^2} \right). \tag{2.1} \end{align}            Write $ \left|1 - c_\alpha \cdot c_{\alpha'} \right|^2 = |1 - t a e^{i \theta}|^2$ with $a e^{i\theta} =  \frac{c_{\alpha}}{|c_{\alpha}|} \cdot    \frac{c_{\alpha'}}{|c_{\alpha'}|} $ and $t = |c_\alpha| |c_{\alpha'}|,$ so that \begin{align} & \left|1 - c_{\alpha} \cdot c_{\alpha'} \right|^2  \nonumber \\ & =  (1 - t)(1 - a^2 t) + t|1 - ae^{i\theta}|^2 \nonumber \\ & =  (1 - |c_\alpha||c_{\alpha'}|) \left(1 - |c_\alpha||c_{\alpha'}|  \left|\frac{c_{\alpha}}{|c_{\alpha}|} \cdot \frac{c_{\alpha'}}{|c_{\alpha'}|}  \right| ^2 \right) + |c_\alpha||c_{\alpha'}| \left|1 -  \frac{c_{\alpha}}{|c_{\alpha}|} \cdot \frac{c_{\alpha'}}{|c_{\alpha'}|} \right|^2 \nonumber \\ & \leq  (1 - |c_\alpha||c_{\alpha'}|)^2 + 2 |c_\alpha | |c_{\alpha'}| (1 - |c_\alpha||c_{\alpha'}|) \left| 1 -  \frac{c_{\alpha}}{|c_{\alpha}|} \cdot  \frac{c_{\alpha'}}{|c_{\alpha'}|} \right| + |c_\alpha||c_{\alpha'}| \left|1 - \frac{c_{\alpha}}{|c_{\alpha}|} \cdot \frac{c_{\alpha'}}{|c_{\alpha'}|} \right|^2 . \tag{2.2} \end{align}

But since \begin{align} (1 - |c_\alpha|^2) \approx (1 - |c_\alpha||c_{\alpha'}|) \approx  e^{-2 \lambda d(\alpha)}, \nonumber \end{align} plugging $(2.2)$ into $(2.1)$, gives us that  \begin{align} e^{2M} & \lesssim  1 + 2   e^{ 2 \lambda d(\alpha)}  \left|1 -  \frac{c_{\alpha}}{|c_{\alpha}|} \cdot \frac{c_{\alpha'}}{|c_{\alpha'}|} \right| +  e^{ 4 \lambda d(\alpha)} \left|1 -  \frac{c_{\alpha}}{|c_{\alpha}|} \cdot \frac{c_{\alpha'}}{|c_{\alpha'}|} \right|^2 \nonumber \\ & =   \left(1 +  e^{ 2 \lambda d(\alpha)} \left|1 -  \frac{c_{\alpha}}{|c_{\alpha}|} \cdot \frac{c_{\alpha'}}{|c_{\alpha'}|} \right|\right)^2 \nonumber\end{align}  so that \begin{align} \left( e^M - 1\right)  e^{- 2 \lambda d(\alpha)}  \lesssim  \left|1 -  \frac{c_{\alpha}}{|c_{\alpha}|} \cdot \frac{c_{\alpha'}}{|c_{\alpha'}|} \right| \nonumber \end{align} which proves the lemma.  \end{proof}

Now, if $\varrho > 0$ and $z \in \mathbb{B}_n$, then let \begin{align} V_{z} ^\varrho = \big\{ w \in \mathbb{B}_n : \left| 1 -  w \cdot \frac{z}{|z|}  \right| \leq {\varrho} (1 - |z|) \big\}.\nonumber \end{align} If $ \varrho = 1$, then we write $V_{z}$ for $V_{z} ^\varrho  $. Also, for $\alpha  \in \mathcal{T}_n$, set $N_\alpha ^R = \{ \omega \in \mathcal{T}_n : D(c_\alpha, R) \cap K_\omega \neq \emptyset \}$ (where ``$N$" stands for neighbors.)  By an easy volume count, there exists $M_R \in \mathbb{N}$ independent of $\alpha$ where $\card N_\alpha ^R \leq M_R$.

\begin{rem} Note that the set $V_{z} ^\varrho$ is a natural higher dimensional version of the classical ``Carleson rectangle'' $R_\varrho (z) \subseteq \mathbb{D}$.  See \cite{Z4}, chap. $5$ for more details. 
\end{rem}

\begin{lem} There exists $\varrho > 0$ independent of $\alpha \in \mathcal{T}_n$ such that $\bigcup_{\beta \geq \alpha} K_\beta \subseteq V_{c_\alpha} ^\varrho$. \end{lem}

\begin{proof}

Let $\beta \geq \alpha$ and let $z \in K_\beta$.  Define a sequence of points $\{z_j\}_{j = 0} ^{d(\beta) - d(\alpha)} \subset \mathbb{B}_n$ and a sequence of tree elements $\{\beta_j\}_{j = 0} ^{d(\beta) - d(\alpha)}$  as follows:   let $z_0 = z, \beta = \beta_0$ and inductively let $z_{j + 1} = P_{\lambda(d(\beta) - j- \frac{1}{2})} (c_{\beta_j})$ for integers $0 < j < d(\beta) - d(\alpha)$ if $z_j \in K_{\beta_j}$ for some $\beta_j \in \mathcal{T}_n$. We will first show that there exists some $R' > 0$ such that $d(z_{d(\beta) - d(\alpha)}, P_{\lambda(d(\alpha) + \frac{1}{2})} z) < R'$ where $R'$ is independent of $\alpha, \beta$ and $z$.

 Note that \begin{align}\left| 1 -   z \cdot c_\beta  \right| ^2 & =  |z| |c_\beta|  \left| 1 -  \frac{z}{|z|} \cdot \frac{c_\beta}{|c_\beta|}  \right|^2 + (1 - |z||c_\beta|)\left(1 - |z||c_\beta|\left| \frac{z}{|z|} \cdot \frac{c_\beta}{|c_\beta|} \right|^2 \right) \nonumber \\ & \geq  |z| |c_\beta|  \left| 1 -  \frac{z}{|z|} \cdot \frac{c_\beta}{|c_\beta|}  \right|^2 . \tag{2.3}  \end{align}  Since $z \in K_\beta$, $(2.3)$ tell us that \begin{align} e^{8\lambda} \geq e^{2d(z, c_\beta)} \gtrsim  \frac{\left| 1 -  z \cdot c_\beta  \right|^2 }{(1 - |c_\beta|^2)^2}   \gtrsim     e^{4\lambda d(\beta) } \left|1 -  \frac{z}{|z|} \cdot \frac{c_\beta}{|c_\beta|}  \right|^2    \nonumber \end{align} so that \begin{align} \left| 1 -   \frac{z}{|z|}  \cdot\frac{c_\beta}{|c_\beta|}  \right| ^\frac{1}{2} \lesssim e^{-\lambda d(\beta) } \tag{2.4} \end{align} if $z \in K_\beta$.

 Using $(2.4)$ in general,  the triangle inequality gives us that \begin{align} \left| 1 -  \frac{z_{d(\beta) - d(\alpha)}}{|z_{d(\beta) - d(\alpha)}|} \cdot \frac{z}{|z|} \right|^\frac{1}{2} &\leq    \sum_{k = 1}^{d(\beta) - d(\alpha)} \left( \left| 1 -  \frac{z_{k-1}}{|z_{k-1}|} \cdot \frac{c_{\beta_{k - 1}} }{|c_{\beta_{k - 1}} | }  \right|^\frac{1}{2} + \left| 1 -  \frac{c_{\beta_{k - 1}} }{|c_{\beta_{k - 1}|}} \cdot \frac{z_k}{|z_k|}   \right|^\frac{1}{2} \right) \nonumber \\ & =   \sum_{k = 1}^{d(\beta) - d(\alpha)}    \left| 1 -  \frac{z_{k-1}}{|z_{k-1}|} \cdot \frac{c_{\beta_{k - 1}} }{|c_{\beta_{k - 1}} | }  \right|^\frac{1}{2} \nonumber \\ & \lesssim   \sum_{k = 1}^{d(\beta) - d(\alpha)} e^{-(d(\beta) - k) \lambda} \nonumber \\ & \lesssim  (1 - |c_\alpha|^2)^\frac{1}{2}. \tag{2.5} \end{align}

Therefore, we have \begin{align} & e^{2d(z_{d(\beta) - d(\alpha)}, P_{\lambda(d(\alpha) + \frac{1}{2})} z))} \nonumber \\  & \leq    \frac{4 \left| 1 -  z_{d(\beta) - d(\alpha)} \cdot P_{\lambda(d(\alpha) + \frac{1}{2})} z) \right|^2}{(1 - |c_\alpha|^2)^2}  \nonumber\\ & \leq   \frac{4 \left( (1 - |c_\alpha|^2)^2 + 2 (1 - |c_\alpha|^2)\left| 1 -  \frac{z_{d(\beta) - d(\alpha)}}{|z_{d(\beta) - d(\alpha)}|} \cdot \frac{z}{|z|}   \right| +  \left| 1 -  \frac{z_{d(\beta) - d(\alpha)}}{|z_{d(\beta) - d(\alpha)}|} \cdot \frac{z}{|z|} \right|^2 \right)}{(1 - |c_\alpha|^2)^2}  \nonumber \\ &   4  \left(1 + (1 - |c_\alpha|^2)^{-1} \left| 1 -  \frac{z_{d(\beta) - d(\alpha)}}{|z_{d(\beta) - d(\alpha)}|} \cdot \frac{z}{|z|}  \right|\right)^2.  \tag{2.6}  \end{align}
Combining $(2.5)$ and $(2.6)$ gives us that $d(z_{d(\beta) - d(\alpha)}, P_{\lambda(d(\alpha) + \frac{1}{2})} z) < R'$ where $R'$ is independent of $\alpha, \beta$ and $z$.

To finish the proof, if $z \in K_\beta$ where $\beta \geq \alpha $ and $z_{d(\beta) - d(\alpha)}$ is defined as before then the triangle inequality tells us that that $d(z_{d(\beta) - d(\alpha)}, c_\alpha ) < 4\lambda$.  Also, we showed that there is an $R' > 0$ independent of $\alpha$ and $\beta$ where $d( z_{d(\beta) - d(\alpha)}, P_{\lambda(d(\alpha) + \frac{1}{2})} z) < R'.$  Thus, we have that \begin{align} & e^{2 d(z, c_\alpha)} \nonumber \\ &  \leq \exp \left(2 d(z, P_{\lambda(d(\alpha) + \frac{1}{2})} z) +  2 d(P_{\lambda(d(\alpha) + \frac{1}{2})} z, z_{d(\beta) - d(\alpha)}) + 2 d(z_{d(\beta) - d(\alpha)} , c_\alpha) \right) \nonumber \\ & \lesssim e^{ 2 \lambda (d(\beta) - d(\alpha))}  \nonumber \end{align} since $d(P_{\lambda(d(\alpha) + \frac{1}{2})} z, z) \lesssim  \frac{1}{2} \ln 4 \left( \frac{1 - |c_\alpha|^2}{1 - |z|^2}\right).$

 This means that \begin{align} e^{2 \lambda (d(\beta) - d(\alpha))}  & \gtrsim  e^{2d(z, c_\alpha)} \nonumber \\ & \gtrsim   \frac{\left| 1 -  z \cdot c_\alpha  \right|^2}{(1 - |z|^2)(1 - |c_\alpha|^2)} \nonumber \\ & \gtrsim       \left| 1 -  z \cdot c_\alpha  \right|^2 e^{2 \lambda (d(\alpha)  + d(\beta))}   \nonumber \end{align} so that $\left| 1  -  z \cdot c_\alpha  \right|^{\frac{1}{2}} \lesssim (1 - |c_\alpha|^2)^{\frac{1}{2}}$.  Using the triangle inequality, we finally have that \begin{align} \left| 1 - z \cdot \frac{c_\alpha}{|c_\alpha|}  \right|^{\frac{1}{2}} & \leq  ( 1 - |c_\alpha|) ^\frac{1}{2} + \left| 1 -  c_\alpha \cdot z \right|^\frac{1}{2} \nonumber \\ & \lesssim  (1 - |c_\alpha|^2)^\frac{1}{2}.\nonumber \end{align}

\end{proof}

\begin{rem} It is not difficult to show that $V_{c_\alpha} \subseteq  \bigcup_{\omega \in N_\alpha ^R } \bigcup_{\beta \geq \omega} K_\beta$ for some $R$ independent of $\alpha$, which gives us a partial converse to the previous lemma.  This says that (as one would expect) the set $\bigcup_{\beta \geq \alpha } K_\beta$ is in some sense ``comparable'' to a Carleson-like subset of $\mathbb{B}_n$.\end{rem}

\section{Discretization of the condition MO${}_\gamma(f) \in L^p(\mathbb{B}_n, d\tau)$}

In this section, we will discretize the condition MO${}_\gamma(f) \in L^p(\mathbb{B}_n, d\tau)$, which will be used to prove Theorem $1.1$.  Note that results and methods of proof in this section are similar to Lemmas $6$ and $7$ (and their proofs) in \cite{X1}

We will need to define two more kinds of sets before we begin.  For each $\alpha \in \mathcal{T}_n$, let ${\widetilde{S}}_\alpha = \bigcup_{\omega \in N_\alpha ^{R}}  \bigcup_{\beta \geq \omega} K_\beta$ where $R > 0$ is a large fixed constant (independent of $\alpha$) that will be determined later.  Also, let $Q_\alpha = \left(\bigcup_{\omega \in N_\alpha ^{6 \lambda}} K_\omega \right)  \cup \left(\bigcup_{\beta \in \mathcal{C} ^1 (\alpha)} K_\beta\right) $.   Note that for all $\gamma > 0$  and  $\beta \in \mathcal{T}_n,$ we have that \begin{align}    v_\gamma(K_\beta) \approx v_\gamma(\widetilde{S}_\beta) \approx v_\gamma (Q_\beta) \nonumber. \end{align}

 Also, for any $f \in L^2(\mathbb{B}_n, dv_\gamma)$ and any Borel set $E \subseteq \mathbb{B}_n$ with $v_\gamma (E)> 0$,  we will let \begin{align} f_E = \frac{1}{v_\gamma(E)} \int_E f  \, dv_\gamma \nonumber \end{align} and let \begin{align} V(f; E) = \left(\frac{1}{v_\gamma (E)} \int_E |f - f_E |^2 \, dv_\gamma\right)^\frac{1}{2}. \nonumber \end{align}

\begin{lem}
If $\frac{2n}{n + 1 + \gamma} < p \leq 1,$  then \begin{align} \int_{\mathbb{B}_n} \{MO_\gamma (f) (z)\}^p d\tau(z) \lesssim  \sum_{\nu \in \mathcal{T}_n} \{V(f; {\widetilde{S}}_\nu
)\}^{p}. \nonumber \end{align}
\end{lem}

\begin{proof}
Since $\tau(K_\alpha)  \approx 1$, it is enough to show that \begin{align} \sum_{\beta \in \mathcal{T}_n} \sup_{z \in K_\beta} \{MO_\gamma (f) (z)\}^p  \lesssim \sum_{\beta \in \mathcal{T}_n} \{V(f; {\widetilde{S}}_\beta
)\}^{p} \nonumber \end{align}

If $\beta \in \mathcal{T}_n$, then for any
$z\in K_\beta, $ it is easy to see that \begin{align} \underset{w \in
\mathbb{B}_n}{\sup} |k_z ^\gamma (w)|^2 \leq  \left(1-|z|^2\right)  ^{-(n + 1 + \gamma)}
 \lesssim  e^{2 \lambda d(\beta) (n + 1 + \gamma) } \lesssim  (v_\gamma (\widetilde{S_\beta}))^{-1}. \nonumber \end{align}  Thus, if $z \in K_o$, then clearly \begin{align} MO_\gamma (f) (z)  \lesssim V(f; {\widetilde{S_o}} )  \nonumber 
\end{align} so we may assume that $d(\beta) > 0$.

 If $d(\beta) > 0$, then let $\nu < \beta$ and  let $ \nu < \tilde{\nu}  \leq \beta $ where $d(\tilde{\nu}) = d(\nu) + 1.$
Let $w \in \mathbb{B}_n \backslash \widetilde{S}_{\tilde{\nu}}$ with $w \in K_\vartheta$, and let $z \in K_\beta$.  We will first show that $\left| 1 -  z \cdot w \right| \gtrsim e^{-2 \lambda (d(\nu) + 1 )} .$ If $d(\vartheta) \leq d(\tilde{\nu})$,  then this is trivial.  Otherwise, assume that $d(\vartheta) > d(\tilde{\nu})$  and assume that $\vartheta \in \mathcal{C} (\nu')$ for some $\nu' $ with $d(\nu') = d(\tilde{\nu})$. By definition, since $w \in \mathbb{B}_n \backslash \widetilde{S}_{\tilde{\nu}}$, we have that $d(c_{\tilde{\nu}}, c_{\nu'}) \geq R$, and so Lemma $2.3$ gives us that \begin{align}\left| 1 -  \frac{c_{\nu'}}{|c_{\nu'}|} \cdot \frac{c_{\tilde{\nu}}}{|c_{\tilde{\nu}}|} \right|^\frac{1}{2} \gtrsim \left(  e^{R} - 1\right)^{\frac{1}{2}} e^{- \lambda ( d(\nu) + 1)}. \nonumber \end{align} By Lemma $2.4$  we have that $\left| 1 -  w \cdot \frac{c_{{\nu'}}}{|c_{{\nu'}}|}  \right| \lesssim (1 - |c_{{\nu'}}|) \lesssim e^{-2 \lambda (d(\nu) + 1)}$ and so an application of the triangle inequality gives us that for some universal constants $C_1 , C_2 > 0$,  \begin{align} \left| 1 -  w \cdot \frac{c_{\tilde{\nu}}}{|c_{\tilde{\nu}}|} \right|^\frac{1}{2} &\geq \left| 1 - \frac{c_{\tilde{\nu}}}{|c_{\tilde{\nu}}|}  \cdot \frac{c_{{\nu'}}}{|c_{{\nu'}}|} \right|^\frac{1}{2}  - \left| 1 -  \frac{c_{{\nu'}}}{|c_{{\nu'}}|} \cdot w  \right| ^\frac{1}{2}  \nonumber \\ & \geq \left[ \left(  e^{R} - 1\right)^{\frac{1}{2}} C_1 - C_2 \right] e^{-\lambda (d(\nu) + 1)}. \nonumber \end{align} As $z \in K_\beta$ with $\beta \geq \tilde{\nu}$, Lemma $2.4$ again gives us a universal constant $C_3$ such that $ \left| 1 -  z  \cdot \frac{c_{\tilde{\nu}}}{|c_{\tilde{\nu}}|}  \right|^\frac{1}{2} \leq C_3 e^{- \lambda (d(\nu) + 1)}$, so that the triangle inequality gives us \begin{align}  \left|1 -  z \cdot w  \right|^\frac{1}{2} & \geq  \left| 1 -  w \cdot \frac{c_{\tilde{\nu}}}{|c_{\tilde{\nu}}|}  \right| ^\frac{1}{2} - \left| 1 -  z \cdot \frac{c_{\tilde{\nu}}}{|c_{\tilde{\nu}}|}  \right|^\frac{1}{2} \nonumber \\ & \geq   \left[ \left(  e^{R} - 1\right)^{\frac{1}{2}} C_1 - C_2 - C_3\right] e^{-\lambda (d(\nu) + 1)} \nonumber \\ & \gtrsim    e^{-\lambda (d(\nu) + 1)} \nonumber \end{align} for any fixed $R$ large enough.

This tells us that \begin{align} |1 -  z \cdot w |^{2(n + 1 + \gamma)}   \gtrsim   e^{- 4 \lambda (d(\nu) + 1)  (n + 1 + \gamma)} \nonumber \end{align} so that for any $z \in K_\beta$, \begin{align}\underset{w \in \mathbb{B}_n \backslash \widetilde{S}_{\tilde{\nu}}} {\sup} |k_z ^\gamma (w)|^2 \lesssim  \frac{e^{-2 \lambda d(\beta) (n + 1 + \gamma)}}{e^{-4 \lambda (d(\nu) + 1) (n  + 1 + \gamma)}} \lesssim (v_\gamma (\widetilde{S}_\nu))^{-1} e^{-2 \lambda (d(\beta) - d(\nu))(n + 1 + \gamma)} .\nonumber
\end{align} \noindent Thus, since $\tilde{S}_ o = \mathbb{B}_n$, $z \in K_\beta$ gives us that \begin{align} \{ & MO_\gamma(f)(z)\}^2 \nonumber \\ & \leq  \int_{\mathbb{B}_n} |f - f_{\widetilde{S}_{\beta}}|^2 |k_z ^\gamma|^2 \, dv_\gamma \nonumber  \\ & \lesssim  \sum_{\nu < \beta} \frac{1}{v_\gamma (\widetilde{S}_\nu)} \int_{\widetilde{S}_\nu \backslash \widetilde{S}_{\tilde{\nu}}} |f - f_{\widetilde{S}_\beta}|^2 e^{-2 \lambda (d(\beta) - d(\nu))(n + 1 + \gamma)} \, dv_\gamma + \{V(f; \widetilde{S}_{\beta})\}^2 \nonumber \\ & \leq  \sum_{\nu < \beta } \frac{1}{v_\gamma (\widetilde{S}_\nu)} \int_{\widetilde{S}_\nu} |f - f_{\widetilde{S}_{\beta}}|^2 e^{-2 \lambda (d(\beta) - d(\nu))(n + 1 + \gamma)} \, dv_\gamma + \{V(f; \widetilde{S}_\beta)\}^2 \nonumber \end{align}

For this fixed $\beta $, let $\nu < \beta$ where $d(\beta) = d(\nu) + \ell$.  Pick $\beta(i)$ for $i \in \{1, \ldots, \ell - 1\}$ where $\beta(i)$ is the parent of $\beta(i + 1)$, and set $\beta(\ell) = \beta$ and $\beta(0) = \nu$.  Then \begin{align} \frac{1}{v_\gamma ({\widetilde{S}}_\nu)} & \int_{{\widetilde{S}}_{\nu}} |f(w) - f_{{\widetilde{S}}_\beta}|^2 \, dv_\gamma (w) \nonumber \\  & \leq  \frac{1}{v_\gamma (\widetilde{S}_{\nu})} \int_{{\widetilde{S}}_{\nu}} \big\{|f(w) - f_{{\widetilde{S}}_{\nu}}| + \sum_{i = 0}^{\ell - 1}|f_{{\widetilde{S}}_{\beta(i)}} - f_{{\widetilde{S}}_{\beta(i + 1)}}| \big\}^2 \, dv_\gamma (w) \nonumber \\ & \lesssim   \{V(f, \widetilde{S}_\nu)\}^2 +  \left(\sum_{i = 0}^\ell V(f; {\widetilde{S}}_{\beta(i)}) \right)^2. \nonumber \end{align} This last inequality follows from the Cauchy-Schwarz inequality and the following general fact: If $E , F \subseteq \mathbb{B}_n$ with $v_\gamma (E \cap F) \neq 0$ then \begin{align} |f_E - f_F | \leq \frac{v_\gamma (E)}{v_\gamma (E \cap F)} V(f; E) + \frac{v_\gamma (F)}{v_\gamma (E \cap F)} V(f; F). \nonumber \end{align}

Thus, for any $0 < p < 1$ and any $z \in K_\beta$ with $\beta > 0$ fixed, we have \begin{align} \{MO_\gamma(f) (z)\}^p & \lesssim   \sum_{\nu \leq \beta} e^{- p \lambda (d(\beta) - d(\nu))((n + 1 + \gamma))}  \sum_{i = 0}^\ell \{V(f; {\widetilde{S}}_{\beta(i)})\}^{p} \nonumber \\ & =   \sum_{\nu \leq \omega  \leq \beta}    e^{-p \lambda (d(\beta) - d(\nu))((n + 1 + \gamma))}  \{V(f; {\widetilde{S}}_\omega)\}^{p}  \nonumber \\ & =  \sum_{\omega  \leq \beta} \{V(f; {\widetilde{S}}_\omega)\}^{p}  \sum_{\nu \leq \omega }  e^{-p \lambda (d(\beta) - d(\nu)(n + 1 + \gamma))}  \nonumber \\ & \lesssim  \sum_{\omega \leq \beta } \{V(f ; {\widetilde{S}}_\omega)\}^{p}  \sum_{\ell = d(\beta) - d(\omega)}^\infty  e^{- p \lambda \ell (n + 1 + \gamma)}. \nonumber  \end{align}

Therefore, we have that \begin{align} \sum_{\beta \in \mathcal{T}_n} \underset{z \in
K_\beta}{\sup} & \{MO_\gamma (f)(z)\}^p \nonumber \\ & \lesssim   \sum_{\beta \in \mathcal{T}_n} \sum_{\omega \leq \beta} \{V(f; {\widetilde{S}}_\omega)\}^{p} e^{-p \lambda (n + 1 + \gamma)   (d(\beta) - d(\omega))} \nonumber \\ & =  \sum_{\omega \in \mathcal{T}_n} \{V(f; {\widetilde{S}}_\omega)\}^{p} \sum_{\ell = 0}^\infty e^{-p \lambda (n + 1 + \gamma )  \ell } \times \card \mathcal{C}^\ell (\omega). \nonumber \end{align}  But by lemma $2.2,
 \card \mathcal{C}^\ell (\omega) \lesssim e^{2 \lambda \ell n} $, and since $(n + 1 + \gamma ) p - 2n > 0$, we have that  \begin{align} \sum_{\beta \in \mathcal{T}_n} \underset{z \in
K_\beta}{\sup} \{MO_\gamma(f)(z)\}^p & \lesssim    \sum_{\beta \in \mathcal{T}_n} \{V(f ; {\widetilde{S}}_\beta)\}^{p}. \nonumber \end{align}

\end{proof}

\begin{lem}
If $0 < p \leq 1$, then \begin{align}
 \sum_{\nu \in \mathcal{T}_n} \{V(f; {\widetilde{S}}_\nu)\}
^{p} \lesssim \sum_{\nu \in \mathcal{T}_n} \{V(f; Q_\nu)\}^{p} \nonumber \end{align}
\end{lem}

\begin{proof}
 By definition $\widetilde{S}_\nu = \bigcup_{\omega \in N_\nu ^{R}}  \bigcup_{\beta \geq \omega} K_\beta. $  Thus, we have that \begin{align} \{V(f; {\widetilde{S}_\nu})\}^2 & \leq   \frac{1}{v_\gamma({\widetilde{S}}_\nu)} \int_{{\widetilde{S}} _\nu} |f - f_{{{Q}}_\nu}|^2 \, dv_\gamma \nonumber \\ & =  \sum_{\omega \in N_\nu ^{R}} \frac{v_\gamma ({Q_\omega})}{v_\gamma({\widetilde{S}}_\nu)}\left[\frac{1}{v_\gamma (Q_\omega)} \int_{Q_\omega} | f - f_{Q_\nu}|^2 \, dv_{\gamma}\right] \nonumber \\ & + \sum_{\omega \in N_\nu ^{R}} \sum_{\beta > \omega}  \frac{v_\gamma ({{Q}}_\beta)}{v_\gamma ({\widetilde{S}}_\nu)} \left[\frac{1}{v_\gamma ({{Q}}_\beta)} \int_{{{Q}}_\beta} |f - f_{{{Q}}_\nu}|^2 \, dv_\gamma \right]. \nonumber \end{align}

  However,  \begin{align} \frac{v_\alpha(Q_\beta)}{v_\alpha({\widetilde{S}}_\nu)} \lesssim  e^{- 2\lambda (n  + 1 + \gamma)  (d(\beta) - d(\nu))} \nonumber  \end{align}

 so that \begin{align} \{V(f; {\widetilde{S}_\nu})\}^2 & \lesssim   \sum_{\omega \in N_\nu ^{R}} \frac{1}{v_\gamma(Q_\omega)} \int_{Q_\omega} | f - f_{Q_\nu}|^2 dv_{\gamma} \nonumber \\ & + \sum_{\omega \in N_\nu ^{R}} \sum_{\beta > \omega}  e^{- 2\lambda (n + 1 + \gamma)  (d(\beta) - d(\nu))} \frac{1}{v_\gamma ({Q}_\beta)} \int_{{Q}_\beta} |f - f_{Q_\nu}|^2 \, dv_\gamma. \tag{3.1} \end{align}

First we deal with the first term in $(3.1)$.  It is easy to see that $\bigcup_{\omega \in N_\nu ^{R}} K_\omega$ is path connected and that $\{ D(c_\omega, 5\lambda) : \omega \in N_\nu ^{R} \} $ is an open cover of $\bigcup_{\omega \in N_\nu ^{R}} K_\omega $.   Thus, by an easy connectivity argument, for each  $\omega \in N_\nu ^{R}, $ there exists a sequence $\{D(c_{\omega_k}, 5\lambda)\}_{k = 1}^{M_\omega} $ where $c_\omega \in D(c_{\omega_1}, 5\lambda)$, $c_\nu \in D(c_{\omega_{M_\omega}}, 5\lambda)$ and  $D(c_{\omega_j}, 5\lambda) \cap D(c_{\omega_{k}}, 5\lambda) \neq \emptyset$ for $ 1 \leq j, k \leq M_\omega$ if and only if $|j - k| \leq 1$.  This implies that there is a Bergman ball of radius $\lambda$ contained in each $D(c_{\omega_k}, 6\lambda) \cap D(c_{\omega_{k + 1}}, 6\lambda)$ so that $v_{\gamma} (Q_{\omega_k} \cap Q_{\omega_{k+1}}) \approx v_{\gamma} (Q_{\nu})$.

\bigskip

Thus, \begin{align} & \frac{1}{v_\gamma(Q_\omega)} \int_{Q_\omega} | f - f_{Q_\nu}|^2 dv_{\gamma} \nonumber \\  & \leq  \frac{1}{v_\gamma(Q_\omega)} \int_{Q_\omega} \left[ |f - f_{Q_\omega}| + |f_{Q_\omega} - f_{Q_{\omega_1}}| + \cdots + \right. \nonumber \\ & \left. |f_{Q_{\omega_{(M_{\omega} - 1)}}} - f_{Q_{\omega_{M_\omega}}}|  +  |f_{Q_{\omega_{M_\omega}}} - f_{Q_{\nu}}| \right]^2 \, dv_\gamma \nonumber \nonumber \\ &   \lesssim \card N_\nu ^{R} \times \sum_{\omega \in N_\nu ^{R}} \{V(f; Q_\omega)\}^2, \nonumber \end{align} which gives us that \begin{align} \sum_{\omega \in N_\nu ^{R}} \frac{1}{v_\gamma(Q_\omega)} \int_{Q_\omega} | f - f_{Q_\nu}|^2 dv_{\gamma}  \lesssim   \left(\card N_\nu ^{R} \right)^2 \sum_{\omega \in N_\nu ^{R}} \{V(f; Q_\omega)\}^2.\tag{3.2} \end{align}

We now work with the second term in $(3.1)$.  As before, let $\omega \in N_\nu ^{R}$ and $\beta > \omega$ with $d(\beta) = d(\omega) + \ell$.  Pick $\beta(i)$ for $i \in \{1, \ldots, \ell - 1\}$ where $K_{\beta(i)}$ is the parent of $K_{\beta(i + 1)}.$  Set $\beta(\ell) = \beta$ and $\beta(0) = \omega$, so that \begin{align} & \frac{1}{v_\gamma (Q_\beta)} \int_{Q_\beta}  |f(w) - f_{Q_\nu}|^2 \, dv_\gamma (w)  \nonumber \\ &  \leq    \frac{1}{v_\gamma (Q_\beta)} \int_{Q_\beta} \big\{|f(w) - f_{Q_\beta}| + \sum_{i = 0}^{\ell - 1}|f_{Q_{\beta(\ell - i)}} - f_{Q_{\beta(\ell - i - 1)}}|  +  |f_{Q_\omega} - f_{Q_\nu}| \}^2 \, dv_\gamma (w) \nonumber \\  &  \lesssim   \left(\sum_{i = 0}^\ell V(f; Q_{\beta(i)}) + \sum_{\omega \in N_\nu ^{R}} V(f; Q_\omega) \right) ^2.  \tag{3.3} \end{align}

Plugging $(3.2)$ and $(3.3)$ into $(3.1)$, we have that \begin{align} \{V(f; {\widetilde{S}}_\nu)\}^{p}  \lesssim \sum_{\omega \in N_\nu ^{R}} \sum_{\omega \leq \eta \leq \beta} e^{- p \lambda (n  + 1 + \gamma) (d(\beta) - d(\nu)) }  \{V(f; Q_\eta)\}^{p}  \nonumber \end{align} as $\card N_\nu ^{R}$ has an upper bound that only depends on $R$.

Thus, we have \begin{align} \{V & (f; {\widetilde{S}}_\nu)\}^{p}   \nonumber \\ & \lesssim   \sum_{\omega \in N_\nu ^{R}} \sum_{\omega \leq \eta \leq \beta} e^{-p\lambda (d(\beta) - d(\nu)) (n  + 1 + \gamma )  }  \{V(f; Q_\eta)\}^{p} \nonumber \\  & =   \sum_{\omega \in N_\nu ^{R}} \sum_{\eta \geq \omega} \{V(f; Q_\eta)\}^{p}  \sum_{\ell = 0}^\infty e^{-p\lambda  (d(\eta) - d(\nu) + \ell ) (n + 1 + \gamma  ) } \times \card \mathcal{C}^\ell (\eta) \nonumber \\ &\lesssim   \sum_{\omega \in N_\nu ^{R}}  \sum_{\eta \geq \omega} \{V(f; Q_\eta)\}^{p} e^{-p \lambda (d(\eta) - d(\nu)) (n + 1 + \gamma ) }. \nonumber\end{align}

Therefore, we finally get that \begin{align} \sum_{\nu \in \mathcal{T}_n } \{V(f; {\widetilde{S}}_\nu)\}^{p} & \lesssim   \sum_{\nu \in \mathcal{T}_n} \sum_{\omega \in N_\nu ^{R}}  \sum_{\eta \geq \omega} \{V(f; Q_\eta)\}^{p} e^{-p \lambda (d(\eta) - d(\nu)) (n + 1 + \gamma) } \nonumber \\ & \lesssim   \sum_{\eta \in \mathcal{T}_n} \{V(f; Q_\eta) \}^{p}, \nonumber \end{align} since \begin{align} \{(\nu, \omega, \eta)   : \nu \in \mathcal{T}_n, \omega \in N_\nu ^{R}, \eta \geq \omega\} \subseteq  \{(\nu, \omega, \eta)   : \eta \in \mathcal{T}_n, \nu \in N_\omega ^{2R}, \omega \leq \eta\}. \nonumber \end{align} if $R > \lambda$.   \end{proof}

\section{Important lemmas}

In this section, we will prove some important lemmas that will be needed in the proof of Theorem $1.1$.  The first of these lemmas (Lemma $4.2$) will be a crucial ``reverse Cauchy-Schwarz'' type inequality that first appeared in the context of the Fock space in \cite{I}, while the rest of the results in this section will be of a more combinatorial nature.  To prove Lemma $4.2$, we will need the following result from \cite{Z4}.

\begin{lem} Suppose $R > 0$ and $b$ is real.  Then  there exists a constant $C_R > 0$ such that \begin{eqnarray} \left| \frac{(1-  z \cdot u )^b}{(1 - z\cdot v )^b} - 1 \right| \leq C_R d(u, v) \nonumber \end{eqnarray} for all $z, u$ and $v$ in $\mathbb{B}_n$ with $d(u, v) \leq R$.\end{lem}

\begin{lem} For small enough $\lambda > 0$, we have that

\begin{align} \{V(f; Q_\nu)\}^2 \lesssim    \int_{Q_\nu} \left|
\int_{Q_{\nu}} \frac{(f(z) - f(w)) (1 - | c_\nu|^2 )^{-\frac{ n + 1 +
\gamma }{2}}} {(1-  z \cdot w )^{n + 1  +
\gamma }}  \, dv_\gamma(w)\right|
^2 \, dv_\gamma(z)   \nonumber \end{align}
   for any $\nu \in \mathcal{T}_n$.
\end{lem}

\begin{rem} An easy application of the Cauchy-Schwarz inequality shows that the reverse inequality is true.  \end{rem}
\begin{proof}
If  $d(Q_\nu)$ denotes the Bergman diameter of $Q_\nu$, then $d(Q_\nu) \leq C_1 \lambda$ for some constant $C_1 > 0$ that is independent of $\lambda$ (assuming that we initially set $\lambda \leq 1$). For $z, w \in Q_\nu$, write
\begin{align} \frac{1}{(1 -  z \cdot w )^{n + 1 + \gamma}} & = \left[ \frac{(1 -  z \cdot c_\nu )^{n + 1 + \gamma}}{(1 -  z \cdot w )^{n + 1 +
\gamma}} - 1\right] \left[\frac{(1 - |c_\nu|^2)^{n + 1 +
\gamma}}{(1 -  z \cdot c_\nu )^{n + 1 + \gamma}} - 1
\right]\frac{1}{(1 - |c_\nu|^2)^{n + 1 + \gamma}} \nonumber
\\ &+ \left[ \frac{(1 -  z \cdot c_\nu )^{n + 1 +
\gamma}}{( 1 -  z \cdot w )^{n + 1 + \gamma}} - 1 \right]\frac{1}{(1 -
|c_\nu|^2)^{n + 1 + \gamma}} \nonumber \\ &+
\left[\frac{(1 - |c_\nu|^2)^{n + 1 + \gamma}}{(1 -
 z \cdot c_\nu )^{n + 1 + \gamma}} - 1 \right]
\frac{1}{(1 - |c_\nu|^2)^{n + 1 + \gamma}} + \frac{1}{(1 -
|c_\nu|^2)^{n + 1 + \gamma}} \nonumber  \end{align}

Thus, we have that \begin{align}
\frac{(1 - |c_\nu|^2)^{-\frac{ n + 1 + \gamma}{2}}}{(1 -  z \cdot w )^{n + 1 + \gamma}} = (1 + \Gamma_{z, w}) \frac{1}{(1 - |c_\nu|^2)^{\frac{3(n + 1 + \gamma)}{2}}} \nonumber \end{align}  where $|\Gamma_{z, w}| \leq C_2\lambda$ for some $C_2 > 0$ independent of $\lambda$ whenever $z, w \in Q_\nu$. Now fix $\lambda > 0$ where $1 - 8 C_2 \lambda > 0.$

This tells us that \[ \left( \left| \int_{Q_\nu} \frac{f(z) -
f(w)}{(1-|c_\nu|^2)^{\frac{3(n + 1 + \gamma)}{2}}}
\, dv_\gamma(w)\right|
 -    \left| \int_{Q_\nu}  \frac{f(z) -
f(w)}{(1-|c_\nu|^2)^{\frac{3(n + 1 + \gamma)}{2}}}
\Gamma_{z, w}
 \, dv_\gamma(w)\right| \right)^2 \] \[ \geq \left| \int_{Q_\nu} \frac{f(z) - f(w)}{(1-|c_\nu|^2)^{\frac{3(n + 1 +
\gamma)}{2}}} \, dv_\gamma(w)\right|^2\] \[ - 2  \left| \int_{Q_\nu} \frac{f(z) - f(w)}{(1-|c_\nu|^2)^{\frac{3(n + 1 +
\gamma)}{2}}} \, dv_\gamma(w)\right|\left| \int_{Q_\nu}
\frac{f(z) - f(w)}{(1-|c_\nu|^2)^{\frac{3(n + 1 +
\gamma)}{2}}} \Gamma_{z, w}
 \, dv_\gamma(w)\right| \] \begin{align}\geq \left| \int_{Q_\nu} \frac{f(z) - f(w)}{(1-|c_\nu|^2)^{\frac{3(n + 1 +
\gamma)}{2}}} \, dv_\gamma(w)\right|^2 - 2 C_2 \lambda \left(\int_{Q_\nu} \frac{|f(z) - f(w)|}{(1-|c_\nu|^2)^{\frac{3(n + 1 + \gamma)}{2}}} \, dv_\gamma(w)\right)^2 . \tag{4.1}\end{align}

Therefore, the triangle inequality and $(4.1)$ implies that
\[ \int_{Q_\nu}
\left| \int_{Q_\nu} \frac{(f(z) - f(w))(1 - | c_\nu|^2 )^{-\frac{ n  + 1 +
\gamma}{2}}}{(1-  z \cdot w )^{n + 1 +
\gamma}} \, dv_\gamma(w)\right|
^2 \, dv_\gamma(z) \nonumber
\]
\begin{align}\geq  \int_{Q_\nu}\left( \left| \int_{Q_\nu} \frac{f(z) - f(w)}{(1-|c_\nu|^2)^{\frac{3(n + 1 +
\gamma)}{2}}} \, dv_\gamma(w)\right|
 -    \left| \int_{Q_\nu}  \frac{f(z) -
f(w)}{(1-|c_\nu|^2)^{\frac{3( n + 1 + \gamma)}{2}}}
\Gamma_{z, w}
 \, dv_\gamma(w)\right| \right)^2 \, dv_\gamma (z) \nonumber  \\ \geq \int_{Q_\nu}  \left[ \left| \int_{Q_\nu} \frac{f(z) - f(w)}{(1-|c_\nu|^2)^{\frac{3(n + 1 +
\gamma)}{2}}} \, dv_\gamma(w)\right|^2 - 2 C_2 \lambda \left(\int_{Q_\nu}  \frac{|f(z) - f(w)|}{(1-|c_\nu|^2)^{\frac{3(n + 1 + \gamma)}{2}}} \, dv_\gamma(w)\right)^2
\right]\, dv_\gamma(z) \nonumber
\\ \approx  \frac{1}{v_\gamma(Q_\nu)^3}  \int_{Q_\nu}  \left[ \left| \int_{Q_\nu} (f(z) - f(w)) \, dv_\gamma(w)\right|^2 - 2 C_2 \lambda
\left(\int_{Q_\nu} |f(z) - f(w)| \, dv_\gamma(w)\right)^2
\right] \, dv_\gamma(z) \nonumber \\ = \left[\{V(f; Q_\nu)\}^2  - \frac{2 C_2
\lambda}{ v_\gamma(Q_\nu)^3}
\int_{Q_\nu} \left(\int_{Q_\nu} |f(z) - f(w)|
\, dv_\gamma(w)\right)^2 \, dv_\gamma(z)\right]. \nonumber
\end{align}

However, by the Cauchy-Schwarz inequality, \begin{eqnarray}
 \frac{2 C_2 \lambda }{ v_\gamma(Q_\nu)^3}
\int_{Q_\nu} \left(\int_{Q_\nu} |f(z) - f(w)|
\, dv_\gamma(w)\right)^2 \, dv_\gamma(z)
\nonumber \\ \leq \frac{2 C_2 \lambda}{ v_\gamma(Q_\nu)^2}  \int_{Q_\nu} \int_{Q_\nu} |f(z) -
f(w)|^2 \, dv_\gamma(w) \, dv_\gamma(z) \nonumber \\ = 4 C_2 \lambda  \{V(f; Q_\nu)\}^2. \nonumber \end{eqnarray}
Since $0 < 8 C_2 \lambda < 1$, we have our result. \end{proof}

   The next lemma in this section will describe how to decompose $\mathcal{T}_n $  into $N$ subsets $\{\mathcal{T}_n ^l \}_{l = 1}^N$ where for each $1 \leq l \leq N, $ we have that $\nu, \alpha \in \mathcal{T}_n ^l$ with $\nu \neq \alpha$ implies that $c_\nu$ and $c_\alpha$ are ``very far apart."  While the details of decompositions like this are simple and in the literature are usually left to the reader, we will present the details because we will need an explicit bound on $N$.
\begin{lem}

For any large positive integer $M$, there exists a decomposition of  $\mathcal{T}_n $  into $N$ subsets $\{\mathcal{T}_n ^l\}_{l = 1}^N$ such that  $N \lesssim M^{2n + 1}$ and $\nu, \alpha \in \mathcal{T}_n ^l$ with $\nu \neq \alpha$ implies that either $|d(\alpha) - d(\nu)| > M $ if $d(\alpha) \neq d(\nu),$ or   $\beta \left( \frac{c_\alpha}{|c_\alpha|},  \frac{c_\nu}{|c_\nu|}\right) > M e^{-\lambda d(\alpha)}$ if $d(\alpha) = d(\nu)$.

\end{lem}

\begin{proof}  First, for $l \in \{1, \ldots, M\}$, let $\mathbb{N}_l = \{j \in \mathbb{N} : j \equiv l \mod M \}$.  Now suppose that for each $k \in \mathbb{N}$, we can decompose $\{\alpha \in \mathcal{T}_n : d(\alpha) = k\}$ into $N_k$ subsets  ${\{T^k _j\}}_{j = 1}^{N_k}$ such that $\max_k N_k \lesssim M^{2n}$ and $\nu, \alpha \in T^k _j$ with $\nu \neq \alpha$ implies that $\beta \left( \frac{c_\alpha}{|c_\alpha|},  \frac{c_\nu}{|c_\nu|}\right) > M e^{- \lambda d(\alpha)}$.  We will then construct the sets $\{\mathcal{T}_n ^l\}_{l = 1}^N$ as follows:  let $\widetilde{N} = \max_k N_k$ and for $1 \leq l \leq M$ and $1 \leq j \leq \widetilde{N}$, let $ T_{(l, j)} =   \bigcup_{k \in \mathbb{N}_l} T_j ^k $ where $T_j ^k := \emptyset$ if $j > N_k$.   Then clearly the decomposition $\{T_{(l, j)}\}$ for $1 \leq l \leq M$ and $1 \leq j \leq \widetilde{N}$ satisfies the conditions in the statement of the lemma.

To finish the proof, fix some $k$ and for simplicity enumerate $\{\alpha \in \mathcal{T}_n : d(\alpha) = k\}$ as $\{\alpha_i \}_{i = 1}^{\mathcal{N}_k}$. Obviously we may assume that $\mathcal{N}_k > M^{2n}$ since otherwise we are done.    Let $\alpha_1^{(1)} = \alpha_1$ and pick the smallest $\ell$ (if it exists) where $\beta \left( \frac{c_{\alpha_\ell}}{|c_{\alpha_\ell}|},  \frac{c_{\alpha_1}}{|c_{\alpha_1}|}\right) > M e^{- \lambda k}$ and let $\alpha_2 ^{(1)} = \alpha_\ell$.  Inductively, for $l \geq 1$, pick the smallest $\ell$ (if it exists) where $\beta \left( \frac{c_{\alpha_\ell}}{|c_{\alpha_\ell}|},  \frac{c_{\alpha_{i}^{(1)}}}{|c_{\alpha_{i}^{(1)}}|}\right) > M e^{- \lambda k}$ for all $1 \leq i \leq l$ and set $\alpha_{i + 1}^{(1)} = \alpha_\ell$.  Continue like this, and set $\alpha_1 ^{(2)} = \alpha_\ell$ where $\ell$ is the smallest integer (if it exists) where $\alpha_\ell \not \in \{\alpha_i ^{(1)}\}_i$.  Define each $\alpha_{i'}^{(2)}$ similarly except that we require $\alpha_{i'} ^{(2)} \not \in \{\alpha_i ^{(1)}\}_i, $ and in general we require that  $\alpha_{i'} ^{(\ell)} \not \in \{\alpha_i ^{(l)}\}_i$ for $1 \leq l \leq \ell$. Setting $T_j ^ k = \{\alpha_i ^{(j)}\}_i $ for $j = 1, 2, \ldots$ will then complete the proof if there exists $N \in \mathbb{N}$ with $N \lesssim M^{2n}$ such that $\{\alpha_i ^{(j)}\}_i = \emptyset$ if $j > N$.

However, this is easy to show.  If $\{\alpha_i ^{(N)}\}_i$ is non-empty for some $N \in \mathbb{N}$, then pick some $\alpha^{N} \in \{\alpha_i ^{(N)}\}_i$. Clearly by construction there exists points $\alpha^{\ell} \in \{\alpha_i^{(\ell)}\}_i$ where $\beta \left( \frac{c_{\alpha ^{\ell}}}{|c_{\alpha ^{\ell}}|},  \frac{c_{\alpha ^{N}}}{|c_{\alpha ^{N}}|}\right) \leq M e^{- \lambda k}$ for $1 \leq \ell \leq N$.  But since these points are distinct (by construction) and the diameter (in the non-isotropic metric) of each $P_{\partial \mathbb{B}_n} K_\alpha$ is equivalent to $e^{-\lambda k}$, an easy volume-count tells us that $N \lesssim M^{2n}$, which completes the proof.
\end{proof}

For the proof of Theorem $1.1$, we will need another notion of the ``neighboring elements'' of some $\alpha \in \mathcal{T}_n$. In particular, for $\alpha \in \mathcal{T}_n$ and $R > 0$,  let \begin{align} \bdd N_\alpha ^R  = \Big \{\omega \in \mathcal{T}_n : d(\omega) = d(\alpha) \text{ and } \beta \left( \frac{c_\alpha}{|c_\alpha|} , \frac{c_\omega}{|c_\omega|} \right) < R e^{- \lambda d(\alpha) }\Big \} \nonumber \end{align} where as before, $\beta$ is the non-isotropic metric on $\partial \mathbb{B}_n$.  Using this metric instead of the Bergman metric when defining $\bdd N_\alpha ^R$ allows us to obtain the following crucial estimate:

\begin{lem} Let $M$ be a large positive integer and write $\mathcal{T}_n = \bigcup_{l = 1}^N \mathcal{T}_n ^l$ as in Lemma $4.3$.  For any $l \in \{1, \ldots, N\}$ and any $\alpha \in \mathcal{T}_n ^l$,  we have that \begin{align} \card\left((\bdd N_\alpha ^{k + 1} \backslash \bdd N_\alpha ^{k}) \cap \mathcal{T}_n ^l \right) \lesssim \frac{k^{2n - 1}}{M^{2n - 1}} \tag{4.2} \end{align} for any integer $k$ with $0 \leq k \leq \sqrt{2} e^{\lambda d(\alpha))}$. Moreover, we have the estimate \begin{align}  \card \left(\bdd N_\alpha ^{k + 1}  \cap \mathcal{T}_n ^l \right) \lesssim \frac{k^{2n }}{M^{2n }} \nonumber \end{align} for any integer $k$ with $0 \leq k \leq \sqrt{2} e^{\lambda d(\alpha))}.$  \end{lem}

\begin{proof} For any $z \in \partial \mathbb{B}_n$ and $r > 0$, let $B_r (z) \subseteq \partial \mathbb{B}_n$ denote a ball with center $z$ and radius $r$ with respect to the non-isotropic metric on $\partial \mathbb{B}_n$.  First note that if $M$  is large enough, then $\card\left((\bdd N_\alpha ^{k + 1} \backslash \bdd N_\alpha ^{k}) \cap \mathcal{T}_n   ^l \right) = 0$ if $k < \frac{M}{2}$.  Now, if $ k \geq \frac{M}{2}$ and $\nu \in (\bdd N_\alpha ^{k + 1} \backslash \bdd N_\alpha ^{k}) \cap \mathcal{T}_n ^l$, then it is easy to see that \begin{align} B_{\frac{M}{2} e^{-\lambda d(\alpha)}} (c_\nu) \subseteq B_{(k + 1 + \frac{M}{2}) e^{-\lambda d(\alpha)}}(c_\alpha) \backslash B_{(k - \frac{M}{2} ) e^{-\lambda d(\alpha)}}(c_\alpha) \nonumber \end{align}  This, combined with the fact that the balls $B_{\frac{M}{2}  e^{-\lambda d(\alpha)}} (c_\nu)$ are pairwise disjoint for $\nu \in (\bdd N_\alpha ^{k + 1} \backslash \bdd N_\alpha ^{k }) \cap \mathcal{T}_n ^l$ tells us that

\begin{align} M^{2n} e^{-2\lambda n d(\alpha)} & \times \card   \left((\bdd N_\alpha ^{k + 1} \backslash \bdd N_\alpha ^{k}) \cap \mathcal{T}_n ^l \right)  \nonumber \\  &  \lesssim   \sigma (B_{(k + 1 + \frac{M}{2})e^{-\lambda d(\alpha) } } (c_\alpha)) - \sigma (B_{(k - \frac{M}{2}) e^{-\lambda d(\alpha)}} (c_\alpha)) \tag{4.3}.  \end{align} It is well known that $\sigma (B_r (z))$ is independent of $z \in \partial \mathbb{B}_n$.  Thus, if we let $L(r) := \frac{\sigma(B_r(z))}{r^{2n}}$, then $L$ is independent of $z$ and Proposition $5.1.4$ of \cite{R} says that $L(r)$ is bounded. Note that $L(r) \leq L(\sqrt{2})$ when $r > \sqrt{2}$.  Using $(4.3)$, we have that \begin{align} & \card  ((\bdd N_\alpha ^{k + 1}    \backslash \bdd N_\alpha ^{k}) \cap \mathcal{T}_n ^l ) \nonumber \\ &  \lesssim  \frac{( k + 1 + \frac{M}{2})^{2n} \sigma (B_{(k + 1 + \frac{M}{2})e^{-\lambda d(\alpha)}} (c_\alpha)) }{M^{2n} (k + 1 + \frac{M}{2})^{2n} e^{-2 \lambda n d(\alpha)}}  - \frac{(k - \frac{M}{2} )^{2n} \sigma (B_{(k -  \frac{M}{2} ) e^{-\lambda d(\alpha)}} (c_\alpha)) }{ M^{2n} (k - \frac{M}{2} )^{2n} e^{-2 \lambda n d(\alpha)}} \nonumber \\ & = M^{-2n} ( k + 1 + \frac{M}{2})^{2n} L(( k + 1 + \frac{M}{2}) e^{-\lambda d(\alpha)}) -  M^{-2n} (k  - \frac{M}{2} )^{2n} L((k -  \frac{M}{2} )e^{-\lambda d(\alpha)}) \nonumber  	 \\ & \lesssim \frac{k^{2n - 1}}{M^{2n - 1}}  \nonumber \end{align} if we can prove that $L$ is Lipschitz on $(0, \sqrt{2})$, since if this is true, then $k  \leq \sqrt{2} e^{\lambda d(\alpha))}$ gives us that \begin{align}  (k + 1 + & \frac{M}{2})^{2n}  L(( k + 1 + \frac{M}{2}) e^{-\lambda d(\alpha)}) - (k - \frac{M}{2} )^{2n} L((k -  \frac{M}{2} )e^{-\lambda d(\alpha)}) \nonumber \\ & =  \left((k + 1 + \frac{M}{2})^{2n} - (k  - \frac{M}{2} )^{2n}\right)  L(( k + 1 + \frac{M}{2}) e^{-\lambda d(\alpha)})  \nonumber \\ & +  (k  - \frac{M}{2} )^{2n} \left( L(( k + 1 + \frac{M}{2}) e^{-\lambda d(\alpha)})  -  L((k -  \frac{M}{2} )e^{-\lambda d(\alpha)}) \right) \nonumber \\ &\lesssim M k^{2n - 1}. \nonumber \end{align}

Thus, we are left to showing that $L$ is Lipschitz.  This is trivial for $n = 1$ since a direct calculation shows that the non-isotropic metric $\beta$ on $\partial \mathbb{D}$ is just the square root of the arc length metric on $\partial \mathbb{D}$.   If $n > 1$ and $u = x + iy$ for $x, y \in \mathbb{R}$,  then Proposition $5.1.4$ of \cite{R} tells us that \begin{align} L(r) = \frac{n - 1}{\pi} \int_{E'(r)} (2x - r^2)^{n - 2} |u|^{-2n} \, dA(u)  \tag{4.4}  \end{align} for any $0 \leq r < \sqrt{2}$, where $E' (r) = \{  x + iy : 2 x > r^2 \text{ and } x^2 + y^2 > 1 \}$. Now, if we define $L(0)$ by $L(0) := \lim_{r \rightarrow 0^+} L(r) = \frac{\frac{1}{4} \Gamma (n + 1) } {(\Gamma(\frac{n}{2} + 1) )^2}$ (see again Proposition $5.1.4$ of \cite{R}), then it is elementary to check (using $(4.4)$) that $L$ is differentiable on $(0, \sqrt{2}), $ continuous on $[0, \sqrt{2}]$, and that $\frac{dL}{dr}$ is bounded on $(0, \sqrt{2})$, which completes the proof of $(4.2)$. The estimate \begin{align}  \card \left(\bdd N_\alpha ^{k + 1}  \cap \mathcal{T}_n ^l \right) \lesssim \frac{k^{2n } }{M^{2n }} \nonumber \end{align} is proven by an argument that is almost identical to (but is easier than) the one above.
\end{proof}

Finally in this section, given some $\alpha \in \mathcal{T}_n$ and $\nu \in \bdd N_\alpha ^{k}$ for some $k \in \mathbb{N}$, we will construct a finite ``chain'' of elements $\{\omega_1\, \ldots, \omega_N \} \subset \bdd N_\alpha ^{k}$ ``connecting'' $\alpha$ and $\nu$.  Note that this is completely trivial when $n = 1$ since we can simply travel in the angular direction from $\alpha$ to $\nu$ (see figure $3$.)  

\begin{center}
\begin{tikzpicture}[scale = 1.75]
 \tikzstyle{every node}=[font=\large]

\fill[color = black!10] 
(40mm, 0mm) arc (0:45:40mm) -- (45 : 40mm) -- (45: 80mm) ;
\fill[color = black!10] 

 (80mm, 0mm) arc (0:45:80mm) -- ( 0 : 40mm) -- (0: 80mm) ;

\foreach \y in {0, ..., 3}
\draw[thin] (15 * \y : 40mm) -- (15 * \y : 77.5mm);

\foreach \y in {0, ..., 6}
\draw[thin] (7.5 * \y : 60mm) -- (7.5 * \y : 77.5mm);

\foreach \y in {0, ..., 12}
\draw[thin] (3.75 * \y : 70mm) -- (3.75 * \y : 77.5mm);

\foreach \y in {0, ..., 24}
\draw[thin] (1.875 * \y : 75mm) -- (1.875 * \y : 77.5mm);

\foreach \y in {0, ..., 48}
\draw[thin] (.9375 * \y : 77.5mm) -- (.9375 * \y : 78.75mm);

\foreach \y in {0, ..., 96}
\draw[thin] (.46875 * \y : 78.75mm) -- (.46875 * \y : 79.375mm);

\foreach \y in {0, ..., 192}
\draw[thin] (.234375 * \y : 78.75mm) -- (.234375 * \y : 79.375mm);


\foreach \y in {0, ..., 384}
\draw[thin] (.1171875 * \y : 79.375mm) -- (.1171875 * \y : 79.6875mm);


\foreach \y in {0, ..., 768}
\draw[thin] (.05859375 * \y : 79.6875mm) -- (.05859375 * \y : 80mm);

\draw[thin] (40mm,0mm) arc (0:45:40mm);
\draw[thin] (60mm,0mm) arc (0:45:60mm);
\draw[thin] (70mm, 0mm) arc (0:45:70mm);
\draw[thin] (75mm, 0mm) arc (0:45:75mm);
\draw[thin] (77.5mm, 0mm) arc (0:45:77.5mm);
\draw[thin] (78.75mm, 0mm) arc (0:45:78.75mm);
\draw[thin] (79.375mm, 0mm) arc (0:45:79.375mm);
\draw[thin] (79.6875mm, 0mm) arc (0:45:79.6875mm);
\draw[thick] (80mm, 0mm) arc (0:45:80mm);

\draw plot[only marks,mark=*, mark size = 0.1mm] coordinates{(7.75:50mm)};
\draw plot[only marks,mark=*, mark size = 0.1mm] coordinates{(38:50mm)};
\draw (6.175 :50mm) node{$c_\alpha$}  ;
\draw (36.175 : 50mm) node{$c_\nu$} ; 
\draw (41 : 83mm) node {$\partial \mathbb{D}$};

\draw[thick, ->] (54.79mm, 4.79mm ) arc (5: 39 : 54.99898mm) ; 

\end{tikzpicture}
 \smallskip

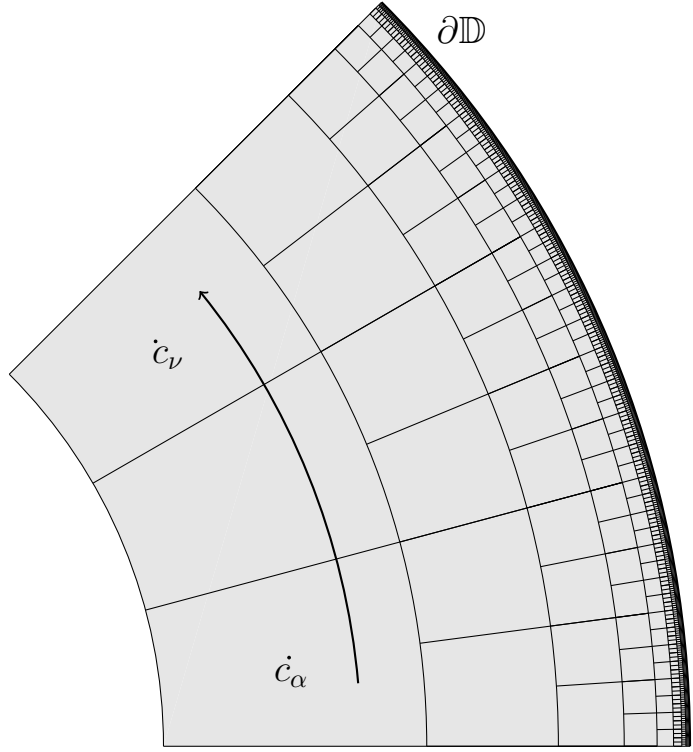
\noindent Figure $3$:  The ``angular'' arrow indicates the tree elements in $\{\eta_{j} ^{\alpha, \nu}\}_{j = 1}^3$ when $n = 1$.  
\end{center}

\begin{lem} Given some $\alpha \in \mathcal{T}_n$ and $\nu \in \bdd N_\alpha ^{k}$ for some $k \in \mathbb{N}$, there exists $\{\eta_{j} ^{\alpha, \nu}\}_{j = 1}^{L} \subset \mathcal{T}_n$ (with $L = L_{(\alpha,\nu)}$) that satisfy the following properties:

\begin{enumerate}
 \item[(i)] $  \eta_{1} ^{\alpha, \nu} = \alpha$ and $ \eta_{L} ^{\alpha, \nu} = \nu$.
\item[(ii)] $ v_\gamma (Q_{\eta_{j} ^{\alpha, \nu}} \cap Q_{\eta_{j + 1} ^{\alpha, \nu}} ) \approx e^{-2\lambda n d(\alpha)}$ whenever $1 \leq j < L$.
\item[(iii)] $ \eta_{j} ^{\alpha, \nu} \neq \eta_{k} ^{\alpha, \nu}$ whenever $1 \leq j, k \leq L$ and $j  \neq k$
\item[(iv)]  There exists $C > 0$ independent of $\alpha$ and $\nu$ such that $\{P_{\partial \mathbb{B}_n} Q_{\eta_{j} ^{\alpha, \nu}}\}_{j = 1}^L \subset \bdd N_\alpha ^{k + C}.$

\end{enumerate}
\end{lem}

\begin{proof}

The construction of the tree elements $\{\eta_{j} ^{\alpha, \nu}\}_{j = 1}^L $ is very similar to the  proof of Lemma $3.4$ in \cite{X2} though we include the details for the sake of completion.

Assume that $c_\alpha$ and $c_\nu$ are linearly independent and write \begin{align} \frac{c_\nu}{|c_\nu|} = \kappa \frac{c_\alpha}{|c_\alpha|} + (1 - |\kappa|^2)^\frac{1}{2}  c^\perp \nonumber \end{align}  where $\kappa$ is a complex number with $|\kappa| < 1 $ and $c^\perp$ is a unit vector in $\mathbb{C}^n$ such that $\frac{c_\alpha}{|c_\alpha|} \cdot c^\perp = 0$. Define the path $p : [0, 1] \rightarrow \partial \mathbb{B}_n$ by \begin{align} p(t) = ((1 - t) +  t \kappa ) \frac{c_\alpha}{|c_\alpha|} + (1 - |(1 - t) +  t \kappa|^2)^\frac{1}{2} c^\perp \nonumber \end{align}   (if $c_\alpha$ and $c_\nu$ are linearly dependent, then define $p(t)$ in the obvious way.)   Note that for each $t \in [0, 1]$, we have that \begin{align} \left|1 - p(t) \cdot \frac{c_\alpha}{|c_\alpha|}\right| = t \left|1 - \frac{c_\nu}{|c_\nu|} \cdot \frac{c_\alpha}{|c_\alpha|} \right| \tag{4.5} \end{align}

Now  for each $\eta \in \mathcal{T}_n$ such that $d(\eta) = d(\alpha)$, let $U_\eta = p^{-1} (P_{\partial \mathbb{B}_n} D(c_\eta, 2\lambda))$ so that $[0, 1] = \cup_\eta U_\eta$ and each $U_\eta$ is open.  We now pick $\eta_{1} ^{\alpha, \nu}, \ldots, \eta_{L} ^{\alpha, \nu}$ as in the proof of Lemma $3.4$ of \cite{X2}: let $ \eta_{1} ^{\alpha, \nu} = \alpha$ and let $t_1 = \sup \{t : t \in U_{\eta_{1} ^{\alpha, \nu}}\}$.  Inductively for $j \geq 1$, if $t_j \not \in U_{\eta_{j} ^{\alpha, \nu}}$, then pick $\eta_{j + 1} ^{\alpha, \nu} $ such that $t_j \in U_{\eta_{j + 1} ^{\alpha, \nu}}$ and let $t_{j + 1} = \sup \{t : t \in U_{\eta_{j + 1}^{\alpha, \nu}} \}$.  Obviously this procedure stops after $L$ steps for some $L = L_{(\alpha, \nu)} \in \mathbb{N}$, and $(i), (ii), (iii)$ are obvious.  Finally, $(iv)$ follows easily from $(4.5)$ and the fact that the diameter of each $P_{\partial \mathbb{B}_n} K_\alpha$ in the non-isotropic metric is equivalent to $e^{- \lambda d(\alpha)}.$

\end{proof}

\section{Proof of Theorem $1.1$}

Finally in this section we will prove Theorem $1.1$. Note that the basic idea behind the proof of Theorem $1.1$ is derived from \cite{I}.  We will need one more elementary result from \cite{Z2} before we begin the proof of Theorem $1.1$.

\begin{lem} Suppose $T$ is a compact operator on a separable Hilbert space $H$ with inner product $\langle \cdot , \cdot \rangle$ and $0 < p \leq 2$.  Then for any orthonormal basis $\{e_n\}$ of $H$, we have that \begin{align} \|T\|_{\s _p} ^p  \leq \sum_{k, j = 1}^\infty |\langle  T e_j, e_k \rangle |^p \nonumber \end{align}

\end{lem}

\textbf{Proof of Theorem} $1.1$: As stated in the introduction, we only need to prove necessity in Theorem $1.1$.  Fix $\lambda > 0$ where lemma $4.2$ is true. Let $\delta = p(n + 1 + \gamma) - 2n$, so that by assumption we have that $\delta > 0$.  Let $M$ be a large integer that will be determined later (and that is allowed to be as large necessary), and write $\mathcal{T}_n = \bigcup_{l = 1}^N \mathcal{T}_n ^l$ as in Lemma $4.3$ and fix any $l \in \{1, \ldots, N\}$. For $M $ large enough, note that the sets $Q_\alpha$ for $\alpha \in \mathcal{T}_n ^l$ are pairwise disjoint.

Let $\{e_\nu\}_{\nu \in \mathcal{T}_n ^l}$ be any fixed arbitrary orthonormal basis for $L^2(\mathbb{B}_n, dv_\gamma)$ (which for no other reason than convenience is chosen to be indexed by the set $\mathcal{T}_n ^l$.)
Set \begin{align} h_\nu  (z) = (1 -
|c_\nu|^2)^{-\frac{ n + 1 + \gamma }{2}} \chi_{{Q}_\nu}(z) \nonumber \end{align}
 and let $\xi_\nu (z) = \frac{ \chi_{{Q}_\nu} (z)  \left([M_f, P_\gamma] h_\nu (z) \right) }{\| \chi_{{Q}_\nu} [M_f, P_\gamma] h_\nu \|}$. Fix some $\rho \in \mathbb{N}$ and let $Z  = \{\beta \in \mathcal{T}_n  :  d(\beta)  \leq \rho\}$. Set $Z_l  = Z  \cap \mathcal{T}_n ^l$ and let $P_{Z_l}$ denote the orthogonal projection onto span $\{e_\nu : \nu \in Z_l\}$.
  Set $W = A ^* [M_f, P_\gamma] B  $ where $A e_\nu = \xi_\nu$ and $B e_\nu = h_\nu $, so that \begin{align} \|W \|_{S_p} ^p \lesssim   \|[M_f, P_\gamma]\|_{S_p} ^ p. \tag{5.1} \end{align}

Clearly we have that $P_{Z_l} W P_{Z_l} f  = \sum_{\nu, \beta \in Z_l} \langle f, e_\nu\rangle \langle W e_\nu, e_{\beta}\rangle e_{\beta}.$  Let $D_l $ be defined by $D_l f = \sum_{\nu \in Z_l} \langle f, e_\nu \rangle \langle W e_\nu, e_\nu \rangle e_\nu $ and set $E_l = P_{Z_l}  W P_{Z_l} - D_l$,  so that \begin{align} \| W \|_{S_p} ^p \geq \|P_{Z_l} W P_{Z_l}\|_{S_p} ^p \geq  \frac{1}{2} \|D_l \|_{S_p} ^p -  \|E_l \|_{S_p} ^p. \tag{5.2} \end{align}

  Since $D_l$ is diagonal, we have that \begin{align} & \| D_l \|_{S_p} ^p =  \sum_{\nu \in Z_l} | \langle  A ^* [M_f, P_\gamma] B  e_\nu , e_\nu  \rangle |^p  \nonumber \\ & =  \sum_{\nu \in Z_l} \|\chi_{Q_\nu} ([M_f, P_\gamma] {h_\nu }) \|^p  \nonumber \\ & =  \sum_{\nu \in Z_l} \left[   \int_{{Q_\nu}} \left|
\int_{{Q_\nu}} \frac{(f(z) - f(w))  (1 - | c_\nu|^2 )^{-\frac{ n + 1 +
\gamma}{2}}}{(1-  z \cdot w )^{n + 1 +
\gamma}}  \, dv_\gamma(w)\right|
^2 \, dv_\gamma(z)\right] ^\frac{p}{2}   \nonumber \\ & \gtrsim   \sum_{\nu \in Z_l} \{V(f; Q_\nu)\}^{p} \nonumber \end{align} where the last inequality is from lemma $4.2$.

Now we need to get an upper bound for $\|E_l \|_{S_p} ^p$.   Since $0 < p \leq 1$, Lemma $5.1$ and the Cauchy-Schwarz inequality (used twice) gives us that  \begin{align} & \|E_l\|_{S_p} ^p    \leq    \sum_{\nu \in Z_l}  \sum_{{\nu '} \in Z_l} | \langle E_l e_\nu, e_{{\nu '}} \rangle |^p \nonumber  \\ & =  \sum_{\nu \in Z_l } \sum_{\begin{subarray}{l} {\nu '} \in Z_l \\ {\nu '} \neq \nu \end{subarray}}  \left(\frac{|\langle [M_f, P_\gamma] h_\nu , \chi_{Q_{{\nu '}}} [M_f, P_\gamma] h_{{\nu '}} \rangle | }{\|\chi_{Q_{{\nu '}}} [M_f, P_\gamma] h_{{\nu '} } \| }\right)^p   \nonumber \\ & \leq  \sum_{\nu \in Z_l } \sum_{\begin{subarray}{l} {\nu '} \in Z_l \\ {\nu '} \neq \nu \end{subarray}} {\|\chi_{Q_{{\nu '}}} [M_f, P_\gamma] h_{\nu} \|^p} \nonumber \\  & =  \sum_{\nu \in Z_l}  \sum_{\begin{subarray}{l} {\nu '} \in Z_l \\ {\nu '} \neq \nu \end{subarray}}   \left[ \int_{Q_{{\nu '}}}\left|
\int_{Q_{\nu}}  \frac{(f(z) - f(w)) (1 - | c_{\nu}|^2 )^{-\frac{ n + 1 +
\gamma}{2}}}{(1-  z \cdot w )^{n + 1 +
\gamma}} \,  dv_\gamma(w)\right|
^2 \, dv_\gamma(z)\right] ^\frac{p}{2}     \nonumber \\ &\lesssim    \sum_{\nu \in Z_l}  \sum_{\begin{subarray}{l} {\nu '} \in Z_l
\\ {\nu '} \neq \nu \end{subarray}}   \left[ \int_{Q_{\nu '}}
\int_{Q_{\nu}}  \frac{|f(z) - f(w)|^2 }{|1-  z \cdot w|^{2(n + 1 +
\gamma)}}  \, dv_\gamma(w)
 \, dv_\gamma(z)\right] ^\frac{p}{2}   \tag{5.3} \nonumber \end{align}

Now assume without loss of generality that $d(\nu') \geq d({\nu })$ and write \begin{align} \mathcal{N} _\nu ^{k} =\bdd N_\nu ^{k  + 1} \backslash \bdd N_\nu ^{k}. \nonumber \end{align}   Moreover,  rewrite  $(5.3)$ as \begin{align}     \sum_{\nu \in Z_l} \sum_{k = 0}^{\lceil \sqrt{2} e^{\lambda d(\nu)}  \rceil} \sum_{\omega \in \mathcal{N} _\nu ^{k} } \sum_{r = 0}^\infty  \sum_{\begin{subarray}{c} {\nu '} \in \mathcal{C}^{r} (\omega)
\\ {\nu '} \neq \nu \end{subarray}}   \left[ \int_{Q_{{\nu '}}}
\int_{Q_{\nu}}  \frac{|f(z) - f(w)|^2 }{|1-  z \cdot w |^{2(n + 1 +
\gamma)}}  \, dv_\gamma(w)
 \, dv_\gamma(z)\right] ^\frac{p}{2}  \tag{5.4} \end{align} where also $\nu' \in Z_l$ in $(5.4).$ Note that if $n = 1$, then $\mathcal{N}_\nu ^k$ can be simply thought of as the two tree elements that are ``$k$ units'' angularly away from $\nu$ (see figure $4$.)   Moreover, going from $(5.3)$ to $(5.4)$ when $n = 1$ is simply the graphically obvious fact that for fixed $\nu$, every $\nu' \geq \nu$ with $\nu' \neq \nu$ is a certain (possibly zero) number of ``units '' $k$ away angularly from $\nu$ and a certain (possibly zero) number of ``units '' $r$ away radially from $\nu$, whereas $\omega$ is the tree element obtained after travelling $k$ ``units`` angularly to the $r^{\text{th}}$ parent of $\nu'$  (again, see figure $4$.)

\begin{center}
\begin{tikzpicture}[scale = 1.75]
 \tikzstyle{every node}=[font=\large]

\fill[color = black!10] 
(40mm, 0mm) arc (0:45:40mm) -- (45 : 40mm) -- (45: 80mm) ;
\fill[color = black!10] 

 (80mm, 0mm) arc (0:45:80mm) -- ( 0 : 40mm) -- (0: 80mm) ;

\foreach \y in {0, ..., 3}
\draw[thin] (15 * \y : 40mm) -- (15 * \y : 77.5mm);

\foreach \y in {0, ..., 6}
\draw[thin] (7.5 * \y : 60mm) -- (7.5 * \y : 77.5mm);

\foreach \y in {0, ..., 12}
\draw[thin] (3.75 * \y : 70mm) -- (3.75 * \y : 77.5mm);

\foreach \y in {0, ..., 24}
\draw[thin] (1.875 * \y : 75mm) -- (1.875 * \y : 77.5mm);

\foreach \y in {0, ..., 48}
\draw[thin] (.9375 * \y : 77.5mm) -- (.9375 * \y : 78.75mm);

\foreach \y in {0, ..., 96}
\draw[thin] (.46875 * \y : 78.75mm) -- (.46875 * \y : 79.375mm);

\foreach \y in {0, ..., 192}
\draw[thin] (.234375 * \y : 78.75mm) -- (.234375 * \y : 79.375mm);


\foreach \y in {0, ..., 384}
\draw[thin] (.1171875 * \y : 79.375mm) -- (.1171875 * \y : 79.6875mm);


\foreach \y in {0, ..., 768}
\draw[thin] (.05859375 * \y : 79.6875mm) -- (.05859375 * \y : 80mm);

\draw[thin] (40mm,0mm) arc (0:45:40mm);
\draw[thin] (60mm,0mm) arc (0:45:60mm);
\draw[thin] (70mm, 0mm) arc (0:45:70mm);
\draw[thin] (75mm, 0mm) arc (0:45:75mm);
\draw[thin] (77.5mm, 0mm) arc (0:45:77.5mm);
\draw[thin] (78.75mm, 0mm) arc (0:45:78.75mm);
\draw[thin] (79.375mm, 0mm) arc (0:45:79.375mm);
\draw[thin] (79.6875mm, 0mm) arc (0:45:79.6875mm);
\draw[thick] (80mm, 0mm) arc (0:45:80mm);

\draw plot[only marks,mark=*, mark size = 0.1mm] coordinates{(7.75:50mm)};
\draw plot[only marks,mark=*, mark size = 0.1mm] coordinates{(38:50mm)};
\draw plot[only marks,mark=*, mark size = 0.1mm] coordinates{(39.375:72.5mm)};
\draw (6.175 :50mm) node{$c_\nu$}  ;
\draw (36.175 : 50mm) node{$c_\omega$} ; 
\draw (38.175 : 72.5mm) node{$c_{\nu'}$} ; 
\draw (41 : 83mm) node {$\partial \mathbb{D}$};

\end{tikzpicture}
 \smallskip

\noindent Figure $4$:  $\nu \in \mathcal{T}_n, \omega \in \mathcal{N}_\nu ^k$, and $\nu' \in \mathcal{C} ^r (\omega)$ for $k = r = 2$ when $n = 1$.
\end{center}

\noindent  Given $\omega \in \mathcal{N} _\nu ^{k}$, let $\{\eta_j ^{\nu, \omega}\}_{j = 1}^{L}$ (with $L = L_{(\nu, \omega)}$) be the ``chain'' of tree elements as in Lemma $4.5$.   Now if $\mathcal{P} ^j (\nu')$ is the $j^\text{th}$ generation parent of $\nu'$, then we can estimate $|f(z) - f(w)|^2$ in $(5.4)$ as \begin{align} |f(z) - & f(w)|^2 \nonumber \\ & \lesssim  |f(w) - f_{Q_{\nu}}|^2 +  \left(\sum_{j = 1}^{L} V(f; Q_{\eta_j ^{\nu, \omega}})\right) ^2  \nonumber \\ & + \left(\sum_{j' = 0}^{r} V(f ; Q_{\mathcal{P} ^j (\nu')})  \right) ^2 + |f(z) - f_{Q_{\nu'}}|^2  \tag{5.5} \nonumber \end{align} (see figure $5$ for the case $n = 1$, where for graphical sake we just assume that $\nu, \nu', \omega \in Z$ for some fixed $\rho$.) 
\begin{center}
\begin{tikzpicture}[scale = 1.75]
 \tikzstyle{every node}=[font=\large]

\fill[color = black!10] 
(40mm, 0mm) arc (0:45:40mm) -- (45 : 40mm) -- (45: 80mm) ;
\fill[color = black!10] 

 (80mm, 0mm) arc (0:45:80mm) -- ( 0 : 40mm) -- (0: 80mm) ;

\foreach \y in {0, ..., 3}
\draw[thin] (15 * \y : 40mm) -- (15 * \y : 77.5mm);

\foreach \y in {0, ..., 6}
\draw[thin] (7.5 * \y : 60mm) -- (7.5 * \y : 77.5mm);

\foreach \y in {0, ..., 12}
\draw[thin] (3.75 * \y : 70mm) -- (3.75 * \y : 77.5mm);

\foreach \y in {0, ..., 24}
\draw[thin] (1.875 * \y : 75mm) -- (1.875 * \y : 77.5mm);

\foreach \y in {0, ..., 48}
\draw[thin] (.9375 * \y : 77.5mm) -- (.9375 * \y : 78.75mm);

\foreach \y in {0, ..., 96}
\draw[thin] (.46875 * \y : 78.75mm) -- (.46875 * \y : 79.375mm);

\foreach \y in {0, ..., 192}
\draw[thin] (.234375 * \y : 78.75mm) -- (.234375 * \y : 79.375mm);


\foreach \y in {0, ..., 384}
\draw[thin] (.1171875 * \y : 79.375mm) -- (.1171875 * \y : 79.6875mm);


\foreach \y in {0, ..., 768}
\draw[thin] (.05859375 * \y : 79.6875mm) -- (.05859375 * \y : 80mm);

\draw[thin] (40mm,0mm) arc (0:45:40mm);
\draw[thin] (60mm,0mm) arc (0:45:60mm);
\draw[thin] (70mm, 0mm) arc (0:45:70mm);
\draw[thin] (75mm, 0mm) arc (0:45:75mm);
\draw[thin] (77.5mm, 0mm) arc (0:45:77.5mm);
\draw[thin] (78.75mm, 0mm) arc (0:45:78.75mm);
\draw[thin] (79.375mm, 0mm) arc (0:45:79.375mm);
\draw[thin] (79.6875mm, 0mm) arc (0:45:79.6875mm);
\draw[thick] (80mm, 0mm) arc (0:45:80mm);

\draw plot[only marks,mark=*, mark size = 0.1mm] coordinates{(7.75:50mm)};
\draw plot[only marks,mark=*, mark size = 0.1mm] coordinates{(38:50mm)};
\draw plot[only marks,mark=*, mark size = 0.1mm] coordinates{(39.375:72.5mm)};
\draw (6.175 :50mm) node{$c_\nu$}  ;
\draw (36.175 : 50mm) node{$c_\omega$} ; 
\draw (38.175 : 72.5mm) node{$c_{\nu'}$} ; 
\draw (41 : 83mm) node {$\partial \mathbb{D}$};

\draw[thick, ->] (54.79mm, 4.79mm ) arc (5: 39 : 54.99898mm) ; 
\draw[thick, <-] (42.13mm, 35.35mm) -- (55.92mm, 46.92mm); 

\end{tikzpicture}
 \smallskip

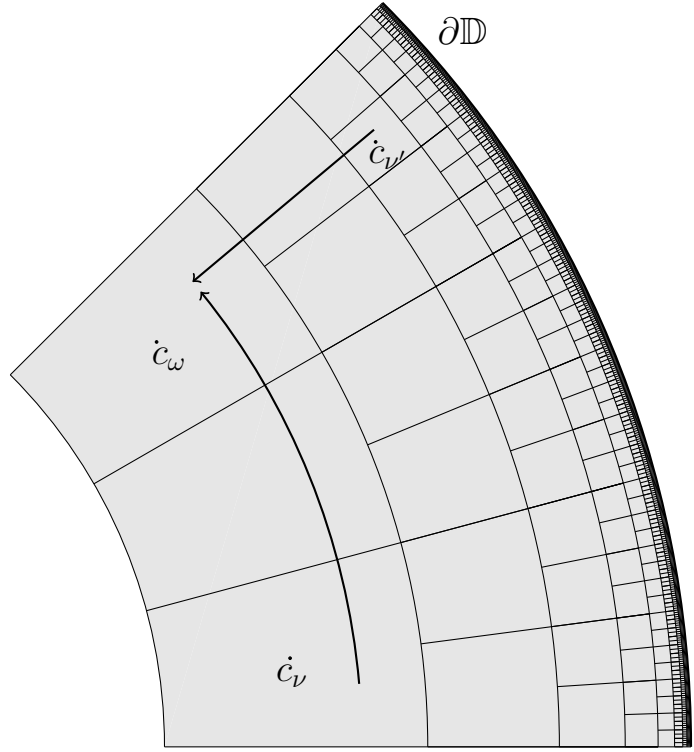
\noindent Figure $5$: $\nu \in Z, \omega \in \mathcal{N}_\nu ^k$, and $\nu' \in \mathcal{C} ^r (\omega)$ for $k = r = 2$ when $n = 1$. The ``angular'' arrow indicates the tree elements in $\{\eta_{j} ^{\alpha, \nu}\}_{j = 1}^3$ and the ``radial'' arrow indicates the tree elements in $\{\mathcal{P}^{j'} (\nu')\}_{j' = 0}^2$.  
\end{center}

 Also the triangle inequality combined with inequalities $(2.3)$ and $(2.5)$ gives us that \begin{align} |1 - z \cdot w|^{2(n + 1 + \gamma)} \gtrsim (k + 1) ^{4(n + 1 + \gamma)} e^{- 4\lambda (n + 1 + \gamma) d(\nu)} \nonumber  \end{align} if $z \in Q_{\nu'}$ and $w \in Q_\nu$,  so that \begin{align} & \frac{v_\gamma (Q_\nu)  v_\gamma (Q_{\nu'})}{| 1 - z \cdot w|^{2(n + 1 + \gamma)}} \nonumber \\ & \lesssim e^{-2 \lambda  (n + 1 + \gamma) d(\nu)} e^{-2 \lambda (n + 1 + \gamma) d(\nu' )} (k + 1) ^{-4(n + 1 + \gamma)} e^{ 4\lambda (n + 1 + \gamma) d(\nu)} \nonumber \\  & = (k + 1) ^{-4(n + 1 + \gamma)} e^{-2 \lambda r  (n +1 + \gamma)} \tag{5.6} \end{align}

After plugging $(5.5)$ and $(5.6)$ into $(5.4)$, we have that \begin{align} & \|E_l\|^p _{\text{S}_p}  \nonumber \\ & \lesssim \sum_{\nu \in Z_l} \sum_{k = 0}^{\lceil \sqrt{2} e^{\lambda d(\nu)}  \rceil} \sum_{\omega \in \mathcal{N} _\nu ^{k} } \sum_{r = 0}^\infty  \sum_{\begin{subarray}{c} {\nu '} \in \mathcal{C}^{r} (\omega) \\ {\nu '} \neq \nu \end{subarray}}  \sum_{j = 0}^{r }  \{V(f ; Q_{\mathcal{P} ^j (\nu')})\} ^{p}  (k + 1) ^{-2p (n + 1 + \gamma)} e^{-p \lambda r  (n +1 + \gamma)}  \nonumber \\ & + \sum_{\nu \in Z_l}  \sum_{k = 0}^{\lceil \sqrt{2} e^{\lambda d(\nu)}  \rceil} \sum_{\omega \in \mathcal{N} _\nu ^{k} } \sum_{r = 0}^\infty  \sum_{\begin{subarray}{c} {\nu '} \in \mathcal{C}^{r} (\omega) \\ {\nu '} \neq \nu \end{subarray}}  \sum_{j = 1}^{L} \{V(f; Q_{\eta_j ^{\nu, \omega}})\} ^{p} (k + 1) ^{-2p (n + 1 + \gamma)} e^{-p \lambda r (n +1 + \gamma)} \nonumber  \\ &= (5.7) + (5.8)  \nonumber \end{align} where again in both $(5.7)$ and $(5.8)$ we assume that $\nu' \in Z_l$.

 To estimate $\|E_l\|^p _{\text{S}_p} $, we will need to deal with four separate terms: $(5.7)$ and $(5.8)$ when $r = 0$ and both of these terms when $r > 0$.  We will first estimate the sum $(5.7)$. The basic idea in estimating all four of these terms will be to switch the order of summation, which will allow us to use Lemma $4.4$ (or a slight variation of Lemma $4.4$) and subsequently obtain the necessary estimates.   If $r = 0$, then $\nu' \in Z_l$ and $\nu \neq \nu' $, which implies that $\beta(c_\nu, c_{\nu'}) > M e^{- \lambda d(\nu)}.$ Thus, $(5.7) $ with $r = 0$ reduces to \begin{align}
\sum_{\nu \in Z_l} &  \sum_{k = M - 1}^{\lceil \sqrt{2} e^{\lambda d(\nu)}  \rceil} \sum_{\begin{subarray}{c} \nu' \in \mathcal{N} _\nu ^{k} \\ {\nu '} \neq \nu \end{subarray} }  \{V(f ; Q_{\nu'})\} ^{p}  (k + 1) ^{-2p (n + 1 + \gamma)}  \nonumber \\ & =  \frac{1}{ M^{2n - 1}} \sum_{\nu' \in Z_l} \{V(f ; Q_{\nu'})\} ^{p} \sum_{k = M - 1}^{\lceil \sqrt{2} e^{\lambda d(\nu')} \rceil} (k + 1) ^{- 2p (n + 1 + \gamma)} (k + 1)^{2n - 1}  \nonumber \\ &  \leq \frac{1}{M^{2n - 1}} \sum_{\nu' \in Z_l} \{V(f ; Q_{\nu'})\} ^{p} \sum_{k = M - 1}^{\lceil \sqrt{2} e^{\lambda d(\nu')} \rceil}   (k + 1)^{-2n - 1 - 2\delta}  \tag{5.9}\end{align} where we have used Lemma $4.4$ in the second to last inequality  and where used the fact that $-\delta = 2n - p (n + 1 + \gamma)$ in the last inequality. The ``integral test'' of elementary calculus now tells us that \begin{align} (5.9) \lesssim M^{-4n - 2\delta +  1 }   \sum_{\nu' \in Z_l} \{V(f ; Q_{\nu'})\} ^{p}. \tag{5.10} \end{align}

Now if $r > 0$, then $\nu \neq \nu'$ and $\nu, \nu' \in Z_l$ tells us that $d(\nu') - d(\nu) \geq M $, so that in this case \begin{align} &(5.7) \nonumber \\ & \lesssim   \sum_{\nu \in Z_l}  \sum_{k = 0}^{\lceil \sqrt{2} e^{\lambda d(\nu)}  \rceil} \sum_{\omega \in \mathcal{N} _\nu ^{k} } \sum_{r = M }^\infty \sum_{{\nu '} \in \mathcal{C}^{r} (\omega)  \cap Z_l}  \sum_{j = 0}^{r } \{V(f ; Q_{\mathcal{P} ^j (\nu')})\} ^{p}  \nonumber \\ & \times (k + 1) ^{-2p (n + 1 + \gamma)} e^{-p \lambda r  (n +1 + \gamma)} \tag{5.11}  \end{align} However, the above sum is taken  over all $0 \leq j \leq r < \infty$ with $r \geq M,   1 \leq k \leq \lceil \sqrt{2} e^{\lambda d(\nu)}  \rceil,  (\nu, \nu') \in Z_l \times Z_l,  \omega = \mathcal{P}^r (\nu')$  where $d(\nu') \geq r,$ and $\nu \in \mathcal{N}^k _{\omega}$.  Thus, we have that \begin{align} &(5.11) \nonumber \\ & \lesssim \sum_{j = 0}^\infty \sum_{r = \max \{M, j\} }^\infty e^{-p \lambda r  (n +1 + \gamma)}\nonumber \\ & \times   \sum_{\begin{subarray}{c} \nu' \in Z_l \\ d(\nu') \geq r  \end{subarray} }\{ V(f ; Q_{\mathcal{P} ^j (\nu') } )\}^{p}   \sum_{k = 0 }^{\lceil \sqrt{2} e^{\lambda d(\nu)}  \rceil} (k + 1) ^{- 2p (n + 1 + \gamma)}  \card \mathcal{N}^k _{\mathcal{P}^r (\nu')} \nonumber \end{align}  By an argument similar to the proof of Lemma $4.4$, we have that \begin{align} \card \mathcal{N}^k _{\mathcal{P}^r (\nu')} \lesssim k^{2n - 1}. \nonumber  \end{align} Moreover, Lemma $2.2$ tells us that \begin{align} \sum_{\begin{subarray}{c} \nu' \in Z_l \\ d(\nu') \geq r  \end{subarray} } \{V(f ; Q_{\mathcal{P} ^j (\nu') } )\}^{p}   \lesssim e^{2n j \lambda } \sum_{\nu' \in Z_l} \{V(f ; Q_{\nu'} )\} ^{p}.   \nonumber \end{align} Thus, if $0 < \epsilon < 1$ where \begin{align} p(1 - \epsilon)(n + 1 + \gamma) > 2n, \nonumber \end{align} then we have that \begin{align}
&(5.11) \nonumber \\  & \lesssim e^{-p\lambda \epsilon  M (n + 1 + \gamma)} \sum_{j = 0}^\infty e^{\lambda j (2n - (1 - \epsilon)p (n + 1 + \gamma))}  \sum_{k = 0 }^{\infty}
(k + 1) ^{- 2n - 1} \sum_{\nu' \in Z_l} \{V(f ; Q_{\nu'} )\} ^{p} \nonumber \\ & \lesssim e^{-p\lambda \epsilon  M (n + 1 + \gamma)} \sum_{\nu' \in Z_l} \{V(f ; Q_{\nu'} )\} ^{p} \tag{5.12} \end{align}

Now we will estimate $(5.8)$, starting again with $r = 0$.  In this case, $(5.8)$ reduces to \begin{align} \sum_{\nu \in Z_l}  \sum_{k = M - 1}^{\lceil \sqrt{2} e^{\lambda d(\nu)}  \rceil} \sum_{\begin{subarray}{c} \nu' \in \mathcal{N} _\nu ^{k} \cap  Z_l  \\ {\nu '} \neq \nu \end{subarray}}  \sum_{j = 1}^{L} \{V(f; Q_{\eta_j ^{\nu, \nu'}})\} ^{p} (k + 1) ^{-2p (n + 1 + \gamma)}. \tag{5.13} \end{align} For each $\nu \in Z_l$, pick some $\nu^k \in  \mathcal{N} _\nu ^{k} \cap  Z_l$ where \begin{align}  \sum_{j = 1}^{L_{\nu, \nu'}} \{V(f; Q_{\eta_j ^{\nu, \nu'}})\} ^{p} \leq  \sum_{j = 1}^{L_{\nu, \nu^k}} \{V(f; Q_{\eta_j ^{\nu, \nu^k}})\} ^{p} \nonumber \end{align}  for any $\nu' \in \mathcal{N} _\nu ^{k} \cap  Z_l$.  Thus, plugging the last inequality into $(5.13)$ and using Lemma $4.4$, we get that \begin{align} (5.13) &  \lesssim  \frac{1}{M^{2n - 1}} \sum_{\nu \in Z_l}   \sum_{k = M-  1}^{\lceil \sqrt{2} e^{\lambda d(\nu)}  \rceil}  k^{2n - 1} k ^{-2p (n + 1 + \gamma)}  \sum_{j = 1}^{L_{\nu, \nu^k}} \{V(f; Q_{\eta_j ^{\nu, \nu^k}})\} ^{p} \nonumber   \end{align}  By $(iv)$ of Lemma $4.5$, we have that  $\eta_j ^{\nu, \nu^k} \in \bdd N_\nu ^{ k + C}$ for some $C$ independent of $k$ and $\nu$. However,  Lemma $4.4$ tells us that \begin{align} \card \left(\bdd N_\omega ^{ k + C}  \cap Z_l \right) \lesssim \frac{k^{2n}}{M^{2n}} \nonumber \end{align} for any $\omega \in \mathcal{T}_n$, and so \begin{align}   & \frac{1}{M^{2n - 1}}  \sum_{\nu \in Z_l}   \sum_{k = M - 1}^{\lceil \sqrt{2} e^{\lambda d(\nu)}  \rceil}  k^{2n - 1} k ^{-2p (n + 1 + \gamma)}   \sum_{\omega \in \bdd N_\nu ^{k + C}} \{V(f; Q_{\omega})\} ^{p} \nonumber \\ & =  \frac{1}{M^{2n - 1}}  \sum_{\omega \in Z}  \{V(f; Q_{\omega})\}^{p} \sum_{k = M - 1}^{\lceil \sqrt{2} e^{\lambda d(\nu)}  \rceil}  k^{2n - 1} k ^{-2p (n + 1 + \gamma)}   \card \left(\bdd N_\omega ^{k + C}  \cap Z_l \right)   \nonumber \\ & \lesssim  \frac{1}{M^{4n - 1}}   \sum_{\omega \in Z}  \{V(f; Q_{\omega})\} ^{p}\sum_{k = M - 1}^{\lceil \sqrt{2} e^{\lambda d(\nu)}  \rceil}  k^{2n - 1} k ^{-2p (n + 1 + \gamma)}  k^{2n}   \nonumber \end{align} which tells us that \begin{align} (5.13) \lesssim M^{-4n - 2 \delta + 1}   \sum_{\omega \in Z}  \{V(f; Q_{\omega})\} ^{p} \tag{5.14} \end{align}

We will finish our estimate of $\|E_l\|_{\text{S}_p} ^p$ by estimating $(5.8)$ when $r > 0$.  As before, if $r > 0$, then $\nu \neq \nu'$ and $\nu, \nu' \in Z_l$ tells us that $d(\nu') - d(\nu) \geq M $, so that in this case we have \begin{align}&  (5.8) \nonumber \\ &  \lesssim
\sum_{\nu \in Z_l}  \sum_{k = 0}^{\lceil \sqrt{2} e^{\lambda d(\nu)}  \rceil} \sum_{\omega \in \mathcal{N} _\nu ^{k} }  \sum_{r = M }^\infty \sum_{\begin{subarray}{c} {\nu '} \in \mathcal{C}^{r} (\omega) \\ {\nu '} \neq \nu \end{subarray}}  \sum_{j = 1}^{L} \{V(f; Q_{\eta_j ^{\nu, \omega}}) \}^{p}  (k + 1) ^{-2p (n + 1 + \gamma)} e^{-p \lambda r (n +1 + \gamma)}  \nonumber \\ &  \lesssim e^{- \frac{M \lambda \delta}{2} } \sum_{\nu \in Z_l}  \sum_{k = 0}^{\lceil \sqrt{2} e^{\lambda d(\nu)}  \rceil} \sum_{\omega \in \mathcal{N} _\nu ^{k} }  \sum_{j = 1}^{L} \{V(f; Q_{\eta_j ^{\nu, \omega}})\} ^{p}  (k + 1) ^{-2p (n + 1 + \gamma)}  \nonumber  \end{align} since $\card{\mathcal{C}^r (\omega)} \lesssim e^{2 \lambda n r}$.  But by an argument that is similar to the proof of Lemma $4.4$, we have that $\card \mathcal{N}_\nu ^k \lesssim k^{2n - 1}$ and $\card \bdd N_\nu  ^k \lesssim k^{2n}$, so that an argument that is almost identical to the estimate of $(5.13)$ gives us that \begin{align} e^{- \frac{ M \lambda \delta}{2} } & \sum_{\nu \in Z_l}  \sum_{k = 0}^{\lceil \sqrt{2} e^{\lambda d(\nu)}  \rceil} \sum_{\omega \in \mathcal{N} _\nu ^{k} }  \sum_{j = 1}^{L} \{V(f; Q_{\eta_j ^{\nu, \omega}})\} ^{p}  (k+1) ^{-2p (n + 1 + \gamma)} \nonumber \\ & \lesssim e^{- \frac{ M \lambda \delta}{2} } \sum_{\omega \in Z} \{V(f; Q_{\omega})\}^{p} \sum_{k = 0}^{\infty} (k+1)^{4n - 1} (k+1)^{- 2p(n + 1 + \gamma)} \nonumber \\ & \lesssim  e^{- \frac{ M \lambda \delta}{2} } \sum_{\omega \in Z} \{V(f; Q_{\omega})\} ^{p}   \tag{5.15} \end{align}

Finally, combining $(5.10), (5.12), (5.14)$ and $(5.15)$, we have that \begin{align} \|E_l\|_{\text{S}_p}^p \lesssim \left(e^{- \epsilon' \lambda  M } + M^{-4n - 2\delta + 1} \right) \sum_{\omega \in Z} \{V(f; Q_{\omega})\} ^{p} \tag{5.16}  \end{align} where $\epsilon' = \min \{p\epsilon(n + 1 + \gamma), \frac{\delta}{2} \}$.

Combining $(5.1), (5.2)$ and $(5.16)$, we get that \begin{align} \|[M_f, & P_\gamma]\|_{\text{S}_p}^p \nonumber \\ & \gtrsim  \frac{1}{2} \|D_l\|_{\text{S}_p}^p - \|E_l\|_{\text{S}_p}^p  \nonumber \\ & \gtrsim \sum_{\nu \in Z_l} \{V(f; Q_\nu)\}^{p}  - \left(e^{- \epsilon' M } + M^{-4n - 2\delta
+  1} \right) \sum_{\nu \in Z} \{V(f; Q_{\nu})\} ^{p} \nonumber \end{align} and summing up over all $1 \leq l \leq N$, we have that \begin{align} M & ^{2n + 1}\|[M_f,  P_\gamma]\|_{\text{S}_p}^p \nonumber \\ &  \gtrsim \sum_{\nu \in Z} \{V(f; Q_\nu)\}^{p}  - N \left(e^{- \epsilon' M } + M^{-4n - 2\delta +  1} \right) \sum_{\nu \in Z} \{V(f; Q_{\nu})\} ^{p} \nonumber \\ & \gtrsim \sum_{\nu \in Z} \{V(f; Q_\nu)\}^{p}  - M^{2n + 1} \left(e^{- \epsilon' M } + M^{-4n - 2\delta + 1}\right) \sum_{\nu \in Z} \{V(f; Q_{\nu})\} ^{p}  \nonumber  \end{align}  since Lemma $4.3$ gives us that $N \lesssim M^{2n + 1}$. Thus, we have that \begin{align} M^{2n + 1} \|[M_f, P_\gamma]\|_{\text{S}_p}^p \gtrsim (1 - o(M))  \sum_{\nu \in Z} \{V(f; Q_{\nu})\} ^{p} \tag{5.16} \end{align} where $\lim_{M \rightarrow \infty} o(M) = 0$.

Finally, recall that $Z = \{\beta \in \mathcal{T}_n : d(\beta) \leq \rho\}$.  Since none of the constants in $(5.16)$ (including $o(M)$) depend on $\rho$, setting $M$ large enough and letting $\rho \rightarrow +\infty$  gives us that \begin{align} \sum_{\nu \in \mathcal{T}_n} \{V(f; Q_{\nu})\} ^{p} \lesssim \|[M_f, P_\gamma]\|_{\text{S}_p}^p \tag{5.17}. \end{align} Combining $(5.17)$ with Lemmas $3.1$ and $3.2$ now completes the proof. \hfill $\square$

\section{The case $0 < p \leq \frac{2n}{n + 1 + \gamma}$.}

Recall from the introduction that MO${}_\gamma (f) \in L^p(\mathbb{B}_n, d\tau)$ when $0 < p \leq \frac{2n}{n + 1 + \gamma}$ if and only if $f$ is constant a.e.  However, a careful reading of the proof of Theorem $1.1$ actually proves the following result:

\begin{thm} Let $f \in \text{BMO}_\partial$ and let $\frac{2n}{ n + 1 + \gamma} < p < \infty$.  For any tree parameter $\lambda$,  the commutator $[M_f, P_\gamma] \in \text{S}_p$ if and only if $\{V(f; Q_\alpha)\}_{\alpha \in \mathcal{T}_n} \in \ell^{p}$  \end{thm}  \noindent (for the case $p \geq 1$, one can check that the results in \cite{Z1, X1} prove Theorem $6.1$.)  Moreover, the following result was proven in \cite{X2}:

\begin{thm}  Let $f \in \text{BMO}_\partial$  and let $\Phi$ be a symmetric gauge function.  For any tree parameter $\lambda > 0$, the commutator $[M_f, P_\gamma] $ is in the symmetrically normed ideal S${}_\Phi$ if and only if $\Phi(\{V(f; Q_\alpha)\}_{\alpha \in \mathcal{T}_n})  < \infty$.  \end{thm}  We refer the reader to the classic text \cite{GK} for the definition and properties of symmetrically normed ideals.  It should be noted that the above theorem was only proved for $\gamma = 0$ and the sets used were slightly different than the sets $Q_\alpha$, though it is easy to see that Theorem $6.2$ holds by using the results in \cite{X2}.  The importance of Theorem $6.2$ for us is that each of the Schatten classes S${}_p$ for $p \geq 1$ are symmetrically normed ideals associated to the symmetric gauge function \begin{align} \Phi(\{a_k\}_{k = 1}^\infty ) = \left(\sum_{k = 1}^\infty |a_k| ^p\right)^\frac{1}{p}. \nonumber \end{align}   In particular, this says that Theorem $6.1$ is true when $1 \leq p \leq \frac{2n}{n + 1 + \gamma}$.  Thus, it is reasonable to conjecture that Theorem $6.1$ is in fact true for all $0 < p \leq \frac{2n}{n + 1 + \gamma}$

To prove sufficiency in Theorem $6.2$, the author in \cite{X2} uses a Riesz functional calculus argument which requires that $\Phi$ define a norm on sequences.  Moreover, to prove necessity, the author uses a simple duality argument involving the symmetrically normed ideal S${}_\Phi$.  Unfortunately, duality is not available when dealing with S${}_p$ if $0 < p < 1$, and clearly $\|\cdot \|_{\ell^p}$ defined on $\ell^p (\mathbb{N})$ is not a norm when $0 < p < 1$, so the basic techniques in \cite{X2} do not work when dealing with S${}_p$ for $0 < p < 1$.

It should also be clear to the reader that the condition $p > \frac{2n}{n + 1 + \gamma}$ was crucial in a number of instances throughout this paper.  Thus, since few of the techniques in this paper or \cite{X2} are applicable in the case where $p \leq  \frac{2n}{n + 1 + \gamma}$ and $p < 1$, it seems like proving Theorem $6.1$ for general $0 < p \leq \frac{2n}{n + 1  + \gamma}$ will be a very difficult task and will require brand new techniques.

Finally, we make some function theoretic remarks when $0 < p < 1$.  In particular, Theorem $6.2$ says that a reasonable definition of the ``Bergman metric Besov space'' B${}_p(\mathbb{B}_n)$ on the ball is the class of all $f \in \text{BMO}_\partial$ where $\{V(f; Q_\alpha)\}_{\alpha \in \mathcal{T}_n} \in \ell^{p}$ for any (or equivalently some) tree parameter $\lambda$.  Note that by the proofs above, we have that $f \in $ B${}_p(\mathbb{B}_n)$ if and only if MO${}_\gamma (f) \in L^p(\mathbb{B}_n, d\tau)$ when $p > \frac{2n}{n + 1 + \gamma}$.  

This raises the question of how to provide a more concrete ``global'' definition of B${}_p(\mathbb{B}_n)$ for arbitrary $0 < p < \infty$.  One simple way to do this (which was precisely done in \cite{X2}) is define the mean oscillation in terms of kernels that have more off-diagonal decay than the normalized reproducing kernels $k_z ^\gamma$ have.  In particular, for $i \in \mathbb{N}$, let \begin{equation*} k_z ^{\gamma, i} = \frac{(1 - |z|^2) ^{\frac{n + 1 + \gamma}{2} + i} }{(1 - w \cdot z) ^{n + 1 + \gamma + i}} \end{equation*} and define \begin{equation*} \widetilde{k_z ^{\gamma, i}} = \frac{k_z ^{\gamma, i}}{\| k_z ^{\gamma, i}\|_{L^2(\mathbb{B}_n, dv_\gamma)}}. \end{equation*}  Note that the standard Rudin-Forelli estimates (see \cite{Z4}, chap. $4$) tell us that $\| k_z ^{\gamma, i}\|_{L^2(\mathbb{B}_n, dv_\gamma)}$ is bounded above and below in $z$, Now define $\text{MO}_{\gamma, i}(f)$ by \begin{equation*} \text{MO}_{\gamma, i}(f) (z) = \|(f - \langle f \widetilde{k_z ^{\gamma, i}} , \widetilde{k_z ^{\gamma, i}} \rangle ) k_z ^{\gamma, i} \end{equation*} where the inner product and norm are taken in $L^2(\mathbb{B}_n, dv_\gamma)$. Clearly we have that \begin{equation*} \text{MO}_{\gamma, i}(f)(z) = \underset{c \in \mathbb{C}}{\inf} \|(f - c) k_z ^{\gamma, i}\|_{L^2(\mathbb{B}_n, dv_\gamma)}. \end{equation*} which means that the arguments in the Section $3$ can be repeated almost word for word to give us that $f \in$ B${}_p(\mathbb{B}_n)$ if and only if $\text{MO}_{\gamma, i}(f) \in  L^p(\mathbb{B}_n, d\tau)$ when $p > \frac{2n}{n + 1 + \gamma + 2i}$.  Note that there are other more classical ways to modify the mean oscillation of a holomorphic function for these purposes (see \cite{Pel} for example.) 


\section*{Acknowledgements} 
The author would like to thank his former advisor Jingbo Xia for interesting discussions regarding this work while he was a graduate student at SUNY Buffalo.  The author would also like to thank the referee for his (or her) helpful remarks.  

\end{document}